\documentclass[10pt,reqno]{amsart}
\usepackage{amscd,amsmath,amssymb,amsthm}
\title{On real extensions of distal minimal homeomorphisms}
\author{Gernot Greschonig}
\address{School of Mathematical Sciences, Tel Aviv University, Ramat Aviv, Tel Aviv 69978, Israel}
\email{greschg@post.tau.ac.il}
\newtheorem{theorem}{Theorem}[section]
\newtheorem*{mainthm}{Main Theorem}

\newtheorem*{cor*}{Corollary}
\newtheorem{lemma}[theorem]{Lemma}
\newtheorem{fact}[theorem]{Fact}

\newtheorem{proposition}[theorem]{Proposition}
\theoremstyle{definition}
\newtheorem{defi}[theorem]{Definition}
\theoremstyle{remark}
\newtheorem{remark}[theorem]{Remark}
\newtheorem{remarks}[theorem]{Remarks}
\newtheorem{remarks*}[]{Remark}
\newtheorem{example}[theorem]{Example}
\newcommand{\comment}[1]{}

\newcommand{\Z}{\mathbb Z}
\newcommand{\T}{\mathbb T}
\newcommand{\R}{\mathbb R}

\newcommand{\mc}{\mathcal}


\begin{document}\allowdisplaybreaks\frenchspacing

\thanks{The author was partially supported by the FWF research grant J-2622 and the research grant 1114/08 of the Israeli Science Foundation}

\begin{abstract}
We prove a structure theorem for topologically conservative real skew product extensions of distal minimal compact metric $\Z$-flows.
The main result states that every such extension can be represented by a perturbation of a Rokhlin skew product.
Moreover, we give certain counterexamples to point out that all components of the construction are in fact inevitable.
\end{abstract}

\maketitle

\section{Introduction}
The study of real-valued topological cocycles and real skew product extensions, also called cylinder flows, has been initiated by Besicovitch, Gottschalk, and Hedlund.
Besicovitch proved the existence of topologically transitive real skew product extensions, and the main result in \cite{G-H}, Chapter 14 can be rephrased to the assertion that a topologically conservative real skew product extension of a minimal rotation on a torus (finite or infinite dimensional) is either topologically transitive or defined by a topological coboundary.
More recently this result has been generalised to skew product extensions of a Kronecker transformation (cf. \cite{LM}).
While the results in \cite{G-H} and \cite{LM} are based on the fact that a minimal rotation acts as an isometry, a corresponding result apart from isometries has been proven for the class of distal minimal homeomorphisms usually called Furstenberg transformations (cf. \cite{Gr}).
However, in general the dichotomy does not hold true that a topologically conservative real skew product extension is either topologically transitive or defined by a topological coboundary.
This motivates the study of topologically conservative real skew product extensions apart from these two cases, carried out here in this paper in the general setting for \emph{distal} minimal homeomorphisms.

Let $T$ be a self-homeomorphism of a compact metric space $(X,d)$, and let $(X,T)$ denote the compact metric $\Z$-flow on $X$ defined by $(n,x)\mapsto T^n x$.
If no indication of the group acting continuously on a compact metric space is given, then in this paper the term compact metric flow will refer to the case of the group $\Z$ of integers with the action being represented by a self-homeomorphism.
We call a flow \emph{minimal}, if the whole space is the only non-empty invariant closed subset of $X$.
If $(X,T)$ and $(Y,S)$ are compact metric flows and $\pi$ is a continuous map from of $X$ \emph{onto} $Y$ with $\pi\circ T=S\circ\pi$, then $(Y,S)$ is called a factor of $(X,T)$ and $(X,T)$ is called an extension of $(Y,S)$.
The set of homeomorphisms commuting with $T$ is a topological group with the topology of uniform convergence, and this group will be denoted by $\textup{Aut}(X,T)$.
Two points $x,y\in X$ are called \emph{distal}, if it holds that
\begin{equation*}
\inf_{n\in\Z}d(T^n x,T^n y)> 0 ,
\end{equation*}
otherwise they are called proximal.
For a general Hausdorff topological group $G$ acting continuously on a compact Hausdorff space $X$ two points $x, y\in X$ are called proximal, if there exists a net $\{g_n\}_{n\in I}\subset G$ with $\lim g_n(x)=\lim g_n(y)$, otherwise they are called distal.
In the case of a compact metric $\Z$-flow these definitions are of course equivalent.
Now a homeomorphism or a group action is called distal, if any two distinct points are distal to each other, and an extension of flows is called distal, if any two distinct points in the same fibre are distal to each other.
One of the most important properties of distal group actions is the \emph{partitioning} of the compact space $X$ into closed minimal subsets, even if the action itself is not minimal.

Suppose that $A$ is an Abelian locally compact second countable (Abelian l.c.s.) group with zero element $\mathbf 0_A$, and let $A_\infty$ denote its one point compactification with the convention that $g+\infty=\infty+g=\infty$ for every $g\in A$.
For a continuous function $f\colon X\longrightarrow A$ we define define a map $f\colon \mathbb{Z}\times X \longrightarrow A$ by
\begin{equation*}
f(n,x)=
\begin{cases}
\sum_{k=0}^{n-1}f(T^k x) & \textup{if}\enspace n\geq 1 ,
\\
\mathbf 0_A & \textup{if}\enspace n=0 ,
\\
-f(-n,T^n x) & \textup{if}\enspace n < 0.
\end{cases}
\end{equation*}
This map satisfies that
\begin{equation*}
f(k,T^l x)+ f(l,x)=f(k+l,x)
\end{equation*}
for all integers $k,l$ and every $x\in X$, and thus it is a \emph{cocycle} of the $\mathbb{Z}$-action on $X$ given by $(n,x)\longmapsto T^n x$.
If a Hausdorff topological group $G$ acts on $X$, then a cocycle is a continuous map $f:G\times X\longrightarrow A$ so that $f(g,h(x))+f(h,x)=f(gh,x)$ holds true for all $g,h\in G$ and $x\in X$.
The \emph{skew product transformation} of the homeomorphism $T$ and the function $f$ is the homeomorphism
\begin{equation*}
\mathbf T_f(x,a)=(T x, f(x)+ a)
\end{equation*}
of $X\times A$, which is related to the cocycle $f(n,x)$ by the equality that
\begin{equation*}
\mathbf T_f^n(x,a)=(T^n x, f(n,x)+ a)
\end{equation*}
for every integer $n$.
We denote the \emph{orbit closure} of a point $x\in X$ under a continuous action of a Hausdorff topological group $G$ on $X$ by
\begin{equation*}
\bar{\mc O}_G(x)=\overline{\{ g(x):g\in G\}}
\end{equation*}
and we denote the orbit closure of $(x,a)\in X\times A$ under $\mathbf T_f$ by
\begin{equation*}
\bar{\mc O}_{T,f}(x,a)=\overline{\{\mathbf T_f^n(x,a):n\in\Z\}}.
\end{equation*}
We call the skew product transformation $\mathbf T_f$ \emph{point transitive} if
\begin{equation*}
\bar{\mc O}_{T,f}(x,a)=X\times A
\end{equation*}
holds true for some point $(x,a)\in X\times A$.
Obviously $(x',a')\in\bar{\mc O}_{T,f}(x,\mathbf 0_A)$ implies that $(x',a'+a)\in\bar{\mc O}_{T,f}(x,a)$ for every $a\in A$, and by the continuity of $\mathbf T_f$ it follows from $(x',a')\in\bar{\mc O}_{T,f}(x,\mathbf 0_A)$ and $(x'',a'')\in\bar{\mc O}_{T,f}(x',\mathbf 0_A)$ that $(x'',a'+a'')\in\bar{\mc O}_{T,f}(x,\mathbf 0_A)$.

Moreover, we shall use the notion of the \emph{prolongation} $\mc D_G(x)$ of a point $x\in X$ under the action of a Hausdorff topological group $G$, which is defined by
 \begin{equation*}
\mc D_G(x)=\bigcap\{\bar{\mc O}_G(\mc U):\mc U\enspace\textup{is an open neighbourhood of}\enspace x\},
\end{equation*}
and the prolongation $\mc D_{T,f}(x,a)$ of a point $(x,a)$ under the skew product transformation $\mathbf T_f$ is then defined by
\begin{equation*}
\mc D_{T,f}(x,a)=\bigcap\{\bar{\mc O}_{T,f}(U):U\enspace\textup{is an open neighbourhood of}\enspace (x,a)\}.
\end{equation*}

While the inclusion of the orbit closure in the prolongation is obvious, the following lemma shows that the coincidence of these sets is actually generic:

\begin{lemma}\label{lem:o_p}
Suppose that $(X,G)$ is a compact metric $G$-flow with a Hausdorff topological group $G$ acting continuously on $X$.
Then there exists a residual $G$-invariant set $\mc F\subset X$ so that
\begin{equation*}
\bar{\mc O}_{G}(x)=\mc D_{G}(x)
\end{equation*}
holds true for every $x\in\mc F$.
Furthermore, for a topological skew product extension $\mathbf T_f:X\times A\longrightarrow X\times A$ of a compact metric flow $(X,T)$ there exists a dense $T$-invariant residual set $\mc F$ of $X$ so that
\begin{equation*}
\bar{\mc O}_{T,f}(x,a)=\mc D_{T,f}(x,a)
\end{equation*}
holds true for every $x\in\mc F$ and $a\in A$.
The assertion holds as well for the extension of $\mathbf T_f$ to the product $X\times A_\infty$ defined by $(x,\infty)\mapsto (Tx,\infty)$ for every $x\in X$.
Moreover, for a continuous function $g=(g_1,g_2):X\longrightarrow\R^2$ the result holds also for the extension of $\mathbf T_g$ onto $X\times(\R_\infty)^2$, which is defined by $(x,s,\infty)\mapsto (Tx,s+g_1(x),\infty)$, $(x,\infty,t)\mapsto (Tx,\infty,t+g_2(x))$, and $(x,\infty,\infty)\mapsto (Tx,\infty,\infty)$, for every $x\in X$ and $s,t\in\R$.
\end{lemma}

\begin{proof}
The first assertion is proven in \cite{AkGl}, the second assertion follows then by application on the extension of $\mathbf T_f$ onto $X\times A_\infty$.
Moreover, the coincidence of $\bar{\mc O}_{T,f}(x,a)$ and $\mc D_{T,f}(x,a)$ for some $(x,a)\in X\times A$ implies the same coincidence for all $(x,a')\in\{x\}\times A$, because $\mathbf T_f$ commutes with the right translation on $X\times A$.
\end{proof}

\begin{defi}
A cocycle $f(n,x)$ is \emph{topologically recurrent} if, for every open neighbourhood $U$ of $\mathbf 0_A$ and every non-empty open set $\mc U\subseteq X$, there exists an integer $n\neq 0$ so that
\begin{equation*}
T^{-n}\mc U \cap \mc U \cap\{x: f(n,x)\in U\}\neq\emptyset.
\end{equation*}
A cocycle which is not topologically recurrent is called \emph{topologically transient}.
\end{defi}

\begin{remarks}\label{rems:rec}
The cocycle $f(n,x)$ is topologically recurrent if and only if the skew product $\mathbf T_f$ is \emph{topologically conservative} (\emph{regionally recurrent} in the terminology of \cite{G-H}), i.e. for every non-empty open set $\mc V \subseteq X\times G$ there exists an integer $n\neq 0$ so that $\mathbf T_f^n(\mc V)\cap\mc V\neq\emptyset$.

If the cocycle $f(n,x)$ is topologically transient and $\{(n_k,x_k)\}_{n\geq 1}\subset\Z\times X$ is a sequence with $d(x_k,T^{n_k} x_k)\to 0$, then it holds true that $\lim_{k\to\infty}|f(n_k,x_k)|\to\infty$.
\end{remarks}

\begin{defi}\label{def:er}
An element $a\in A$ is in the set $E(f)$ of \emph{topological essential values} of the cocycle $f(n,x)$ if, for every open neighbourhood $U(a)$ of $a$ and every non-empty open set $\mc U\subseteq X$, there exists an integer $n\neq 0$ so that
\begin{equation*}
T^{-n}\mc U \cap \mc U \cap\{y: f(n,y)\in U(a)\}\neq\emptyset.
\end{equation*}
The set of topological essential values is called the \emph{topological essential range}.
\end{defi}

\begin{remarks}\label{rem:er}
The cocycle $f(n,x)$ is topologically recurrent if and only if the zero element $\mathbf 0_A$ is an element of $E(f)$.
Moreover, if the topological essential range $E(f)$ is non-empty, then it is a closed \emph{subgroup} of $A$ (cf. \cite{LM}, Proposition 3.1).

A topological cocycle $f(n,x)$ has full topological essential range $E(f)=A$ if and only if the skew product transformation $\mathbf T_f$ is topologically transitive on $X\times A$ (\emph{regionally transitive} in the terminology of \cite{G-H}), i.e. for two arbitrary non-empty open sets $\mc V, \mc W \subseteq X\times A$ there exists an integer $n$ so that $\mathbf T_f^n(\mc V)\cap\mc W\neq\emptyset$ (cf. \cite{LM}, Proposition 3.2).
The statement that $E(f)=A$ is also equivalent to the point transitivity of $\mathbf T_f$, with a residual set of $x\in X$ so that $\bar{\mc O}_{T,f}(x,a)= X\times A$ for every $a\in A$ (cf. \cite{G-H}, Theorem 9.20).
\end{remarks}

\begin{defi}
Let $b:X\longrightarrow A$ be a continuous function and define a topological cocycle $h(n,x)$ by the function $h(x)=f(x)-b(T x)+b(x)$.
Then the cocycle $h(n,x)$ is called \emph{topologically cohomologous} to the cocycle $f(n,x)$ with the \emph{transfer function} $b(x)$, and it follows immediately that
\begin{equation*}
h(n,x)=f(n,x)-b(T^n x)+b(x).
\end{equation*}
A cocycle topologically cohomologous to zero is called a \emph{topological coboundary}.
\end{defi}

The following lemma has been used in Atkinson's proof \cite{A} that recurrent $\R^d$-valued topological cocycles of a minimal rotation on a torus are coboundaries if and only if the essential range is trivial.
In our case a generalised version for cocycles of a \emph{minimal} homeomorphism on a compact metric space is required, and for the sake of simplicity the lemma will be restricted to real valued cocycles.

\begin{lemma}\label{lem:at}
\begin{enumerate}
\item\label{as:s_l}Let $f(n,x)$ be a real valued topological cocycle of a minimal homeomorphism $T$ of a compact metric space $(X,d)$, and suppose that the skew product transformation $\mathbf T_f$ is \emph{not} topologically transitive on $X\times\R$.
Then for every neighbourhood $U$ of $0$ there exist a compact symmetric neighbourhood $K\subset U$ of $0$ and an $\varepsilon>0$ so that the set
\begin{equation}\label{eq:s_l}
\{x\in X:d(x, T ^n x)<\varepsilon\enspace\textup{and}\enspace f(n,x)\in 2K\setminus K^0\}
\end{equation}
is empty for every integer $n$.

\item More general, let $f(n,x)$ be a cocycle with values in an Abelian l.c.s. group $A$ and let $\{(n_k,x_k)\}_{k\geq 1}\subset(\Z\setminus\{0\})\times X$ be a sequence with $d(x_k, T^{n_k} x_k)\to 0$ and $f(n_k,x_k)\to g\in A_\infty$ (respectively $\R_\infty\times\R_\infty$) as $k\to\infty$.
Then for every $x\in X$ it holds true that $(x,g)\in\mc D_{T,f}(x,\mathbf 0_A)$.
Moreover, if $g\in A$ (i.e. finite), then it is an element of the essential range $E(f)$.
\end{enumerate}
\end{lemma}

\begin{proof}
We want to start with the statement (ii).
Let $\{(n_k,x_k)\}_{k\geq 1}\subset(\Z\setminus\{0\})\times X$ be a sequence with the stated properties, and let $\mc U\subset X$ be an arbitrary non-empty open set and $U(g)$ an arbitrary neighbourhood of $g\in A_\infty$.
We may assume that $x_k\to x'\in X$, and as the homeomorphism $T$ is minimal, we can fix an integer $m$ with $T^m x'\in\mc U$.
It follows then that $T^m x_k\to T^m x'$ and $T^{n_k} T^m x_k=T^m T^{n_k} x_k\to T^m x'$ as $k\to\infty$, and we obtain by the cocycle identity and the continuity of $f(m,\cdot)$ that
\begin{eqnarray*}
f(n_k,T^m x_k) & = & f(m,T^{n_k} x_k) + f(n_k,x_k)+f(-m,T^m x_k)\\
& = & f(m,T^{n_k} x_k)+ f(n_k,x_k)-f(m, x_k) \to g.
\end{eqnarray*}
For all $k$ large enough it holds true that $T^m x_k$, $T^{n_k} T^m x_k\in\mc U$, and $f(n_k,T^m x_k)\in U(g)$.
It follows that $(x,g)\in\mc D_{T,f}(x,\mathbf 0_A)$ for every $x\in X$, because the neighbourhoods $\mc U$ and $U(g)$ were arbitrary, and in the case $g\in A$ this implies that $g\in E(f)$.

Suppose that $f(n,x)$ is real valued and $\mathbf T_f$ is not topologically transitive.
Then the essential range $E(f)$ is a proper closed subgroup of $\R$ by Remarks \ref{rem:er}, and for every neighbourhood $U$ of $0$ there exists a compact symmetric neighbourhood $K\subset U$ with $2K\setminus K^0\cap E(f)=\emptyset$.
If assertion (\ref{as:s_l}) is false for the compact symmetric neighbourhood $K$, then the statement (ii) implies that $g\in E(f)\cap 2K\setminus K^0$, in contradiction to the choice of $K$.
\end{proof}

Now we want to define a topological version of the \emph{Rokhlin extension} and the \emph{Rokhlin skew product}.
Moreover, we shall introduce the notion of a \emph{perturbated Rokhlin skew product}, which will be inevitable in the statement of our main result.

\begin{defi}
Suppose that $(X,T)$ is a distal minimal compact metric flow and that $(M,\{\Phi_t:t\in\R\})$ is a distal minimal compact metric $\R$-flow.
Let $f:X\longrightarrow\R$ be a function with a topologically transitive skew product $\mathbf T_f$ on $X\times\R$.
We define the \emph{Rokhlin extension} $T_{\Phi,f}$ on $X\times M$ by
\begin{equation*}
T_{\Phi,f}(x,m)=(T x,\Phi_{f(x)}(m)) ,
\end{equation*}
and we define the \emph{Rokhlin skew product} $\mathbf T_{\Phi,f}$ on $X\times M\times\R$ by
\begin{equation*}
\mathbf T_{\Phi,f}(x,m,t)=(T x,\Phi_{f(x)}(m),t+f(x)) .
\end{equation*}
Now let $g:\R\times M\longrightarrow\R$ be a cocycle of the $\R$-flow $(M,\{\Phi_t:t\in\R\})$.
Then we define the \emph{perturbated Rokhlin skew product} $\mathbf T_{\Phi,f,g}$ on $X\times M\times\R$ by
\begin{equation*}
\mathbf T_{\Phi,f,g}(x,m,t)=(T x,\Phi_{f(x)}(m),t+f(x)+g(f(x),m)) .
\end{equation*}
\end{defi}

\begin{remark}
We obtain from the cocycle identities for $f(n,x)$ and $g(t,m)$ that
\begin{equation*}
T_{\Phi,f}^n(x,m)=(T^n x,\Phi_{f(n,x)}(m)) ,
\end{equation*}
\begin{equation*}
\mathbf T_{\Phi,f}^n(x,m,t)=(T^n x,\Phi_{f(n,x)}(m),t+f(n,x)),
\end{equation*}
and
\begin{equation*}
\mathbf T_{\Phi,f,g}^n(x,m,t)=(T^n x,\Phi_{f(n,x)}(m),t+f(n,x)+g(f(n,x),m)
\end{equation*}
hold true for every integer $n$.
If $(x,0)$ is a transitive point for $\mathbf T_f$, then it follows from the minimality of the flow $(M,\{\Phi_t:t\in\R\})$ that $\{x\}\times M\subset\bar{\mc O}_{T_{\Phi,f}}(x,m)$ for every $m\in M$.
This implies however that $(x,m)$ is a transitive point for the distal homeomorphism $T_{\Phi,f}$, and therefore the compact $Z$-flow $(X\times M,T_{\Phi,f})$ is minimal.
\end{remark}

First we want to present a simple example of a topological Rokhlin skew product which is of topological type $III_0$, i.e. recurrent with a trivial topological essential range but not a topological coboundary.

\begin{example}\label{ex:zi}
Let $f:\T\longrightarrow\R$ be a continuous function with a topologically transitive skew product extension $\mathbf T_f$ of the irrational rotation $T$ by $\alpha$ on the torus, and let $\beta\in(0,1)$ be irrational so that the $\R$-flow $\{\Phi_t:t\in\R\}$ defined by
\begin{equation*}
\Phi_t(y,z)=(y+t,z+\beta t)
\end{equation*}
is minimal and distal on $\T^2$.
Then the minimal and distal Rokhlin extension $T_{\Phi,f}$ on $\T^3$ turns out to be
\begin{equation*}
T_{\Phi,f}(x,y,z)=(x+\alpha,y+f(x),z+\beta f(x)),
\end{equation*}
and putting $h(x,y,z)=f(x)$ for all $(x,y,z)\in\T^3$ gives a topological type $\textup{III}_0$ cocycle $h(n,(x,y,z))$ of the homeomorphism $T_{\Phi,f}$ with the skew product extension $\mathbf T_{\Phi,f}$.
Indeed, as $\mathbf T_f$ is point transitive, the cocycle $h(n,(x,y,z))$ is recurrent, but it is not bounded and therefore cannot be a topological coboundary.
Furthermore, a sequence $\{t_n\}_{n\geq 1}\subseteq\R$ with $t_n\mod 1\to 0$ and $(\beta t_n)\mod 1\to 0$ cannot have a finite cluster point apart from zero, and hence $E(h)=\{0\}$.
For a point $\tilde x\in\T$ so that $(\tilde x,0)\in\T\times\R$ is transitive under $\mathbf T_f$ and for arbitrary $y,z\in\T$ the orbit closure of $(\tilde x,y,z),0)$ under the skew product extension of $T_{\Phi,f}$ by $h$ is of the form
\begin{equation*}
\bar{\mc O}_{\mathbf T_{\Phi,f}}((\tilde x,y,z),0)=\bar{\mc O}_{T_{\Phi,f},h}((\tilde x,y,z),0)=\T\times\{(\Phi_t(y,z),t)\in\T^2\times\R:t\in\R\} .
\end{equation*}
The collection of these sets for all $(y,z)\in\T^2$ defines a partition of $\T^3\times\R$ into orbit closures under $\mathbf T_{\Phi,f}$.
\end{example}

The next example makes clear that the perturbation of a Rokhlin skew product by a cocycle $g (t,m):\R\times M\longrightarrow\R$ of the $\R$-flow $\{\Phi_t:t\in\R\}$ cannot necessarily be eliminated by cohomology with respect to a continuous transfer function.
\begin{example}
We let $T$, $f$, and $\{\Phi_t:t\in\R\}$ be defined as in Example \ref{ex:zi}, and we let $g(t,(y,z))$ be a topologically transitive cocycle for the $\R$-flow $\{\Phi_t:t\in\R\}$.
We put
\begin{equation*}
\tilde h(x,y,z)= f(x)+g(f(x),(y,z))
\end{equation*}
and obtain that
\begin{equation*}
\tilde h(n,(x,y,z))= f(n,x)+g(f(n,x),(y,z))
\end{equation*}
is a cocycle of $T_{\Phi,f}$ with the skew product extension $\mathbf T_{\Phi,f,g}$ on $\T^3\times\R$.
From unique ergodicity it follows that $\int_{\T^2} g(t,(y,z)) d\lambda(y,z)=0$ for every $t\in\R$.
As the perturbation $g(h(n,x),(y,z)))$ is unbounded, there cannot be a continuous cobounding function defined on $\T^3$ making $\tilde h$ and $h$ cohomologous.
However, the structure of the orbit closures in the skew product is preserved in the sense that
\begin{equation*}
\bar{\mc O}_{\mathbf T_{\Phi,f,g}}((\tilde x,y,z),0)=\T\times\{(\Phi_t(y,z),t+g(t,(y,z)))\in\T^2\times\R:t\in\R\} .
\end{equation*}
\end{example}

Motivated by these examples we want to formulate the main result of this paper:

\begin{mainthm}
Suppose that $(X,T)$ is a distal minimal compact metric $\Z$-flow and $f:X\longrightarrow\R$ is a continuous function with a topologically recurrent cocycle which is not a coboundary.
Then there exist a factor $(X_\alpha,T_\alpha)=\pi_\alpha(X,T)$, a continuous function $f_\alpha:X_\alpha\longrightarrow\R$, a compact metric space $(M,\delta)$, and a continuous distal $\R$-flow $\{\Phi_t:t\in\R\}$ on $M$ so that the Rokhlin extension $(X_\alpha\times M,T_{\alpha,\Phi,f_\alpha})$ is also a factor $(Y,S)=\pi_Y(X,T)$ of $(X,T)$ and the function $f$ is topologically cohomologous to $f_Y\circ\pi_Y$ for a suitable continuous function $f_Y:Y\longrightarrow\R$.
Moreover, there exists a topological cocycle $g:\R\times M\longrightarrow\R$ of the $\R$-flow $(M,\{\Phi_t:t\in\R\})$ so that
\begin{equation*}
f_Y(x,m)=f_\alpha(x)+g(f_\alpha(x),m)
\end{equation*}
holds true for every $(x,m)\in Y=X_\alpha\times M$, and thus the skew product $\mathbf S_{f_Y}$ is the perturbated Rokhlin skew product $\mathbf T_{\alpha,\Phi,f_\alpha, g}$.
Moreover, there exists a residual set of $x\in X$ for which it holds true that
\begin{equation*}
\pi_Y^{-1}(x)\times\{0\}\subset\bar{\mc O}_{T, f_Y\circ\pi_Y}(x,0) .
\end{equation*}
\end{mainthm}

\begin{cor*}
The skew product $\mathbf R_{f_\alpha\circ\tau_\alpha}$ over the distal homeomorphism
\begin{eqnarray*}
& R:& Y\longrightarrow Y \\
& & (x,m)\mapsto(T_\alpha x,m) ,
\end{eqnarray*}
in which $\tau_\alpha:(Y,S)\longrightarrow(X_\alpha,T_\alpha)$ denotes the factor map $(x,m)\mapsto x$, is a topologically transitive extension of every minimal $R$-orbit closure $X_\alpha\times\{m\}\subset Y$.
This skew product is related to the skew product transformation $\mathbf S_{f_Y}$ by the continuous mapping $F:Y\times\R\longrightarrow Y\times\R$ defined by
\begin{equation*}
F(x,m,t)=(x,\Phi_t(m),t+g(t,m))
\end{equation*}
so that
\begin{equation*}
F\circ\mathbf R_{f_\alpha\circ\tau_\alpha}=\mathbf S_{f_Y}\circ F .
\end{equation*}

If the minimal compact metric flow $(X,T)$ is uniquely ergodic, then the mapping $F$ is onto and closed.
Therefore the skew product $\mathbf S_{f_Y}$ is a topological factor of the skew product $\mathbf R_{f_\alpha\circ\tau_\alpha}$, and the space $Y\times\R$ admits a partition into the collection of $\mathbf S_{f_Y}$-orbit closures given by the images of the sets $X_\alpha\times\{m\}\times\R$ under the closed continuous onto mapping $F$ for all $m\in M$.
\end{cor*}

\begin{remarks}
Even though the compact metric flow $(X,T)$ is not necessarily a Rokhlin extension, the existence of a function $f:X\longrightarrow\R$ with a topologically recurrent cocycle apart from a coboundary and with a non-transitive skew product extension forces the existence of a Rokhlin extension factor $(Y,S)$ with a non-trivial flow $\{\Phi_t:t\in\R\}$ as well as the existence of a function cohomologous to $f$ defined on this factor.

If the flow $(X,T)$ is uniquely ergodic, then the perturbation cocycle $g(t,y)$ fulfils that $g(t,y)/t\to 0$ uniformly for all $y\in Y$.
Therefore the perturbation, even if not uniformly bounded, is dominated by the linear term in the Rokhlin skew product.
However, even in the non uniquely ergodic case the perturbation cannot establish transitivity of the skew product $\mathbf S_{f_Y}$, because the transitivity of $\mathbf T_f$ follows then.
\end{remarks}

The proofs of our main results will be concluded at the end of the following section, which starts with the structure theory of \emph{distal} minimal flows as our most important tool.

\section{Real cocycles of distal minimal flows}

Furstenberg's structure theorem for distal minimal flows will be essential in the further study of cocycles.
The structure theorem is based on the following definitions of an $M$-bundle and an isometric extension.

\begin{defi}
Let $X$ and $Y$ be compact metric spaces, $\pi$ be a continuous map from of $X$ onto $Y$, and $M$ be a compact homogenous metric space.
Suppose that there exists a real valued function $\rho(x_1,x_2)$ defined on
\begin{equation}
R_\pi=\{(x_1,x_2)\in X\times X: \pi(x_1)=\pi(x_2)\}
\end{equation}
 and continuous on $R_\pi$, so that for every $y\in Y$ the function $\rho$ is a metric on the fibre $\pi^{-1}(y)$ with an isometry between $\pi^{-1}(y)$ and $M$.
Then $X$ is called an $M$-\emph{bundle} over $Y$.
\end{defi}

\begin{defi}
Let $(X,T)$ and $(Y,S)=\pi(X,T)$ be compact metric flows so that $X$ is an $M$-bundle over $Y$.
If the function $\rho$ satisfies that $\rho(x_1,x_2)=\rho(T x_1,T x_2)$ for all $x_1,x_2$ in the same fibre of $X$ over $Y$, then $(X,T)$ is called an \emph{isometric extension} of $(Y,S)$.
\end{defi}

\begin{fact}[Furstenberg's structure theorem]
Let $(X,T)$ be a distal minimal compact metric flow.
Then there exists a countable ordinal $\eta$ with subflows $(X_\xi,T_\xi)=\pi_\xi(X,T)$ for each ordinal $0\leq\xi\leq\eta$, so that the following properties hold true:
\begin{enumerate}
\item $(X_\eta,T_\eta)=(X,T)$ and $(X_0,T_0)$ is the trivial flow.
\item $(X_\xi,T_\xi)=\pi_\xi^{\zeta}(X_{\zeta},T_{\zeta})$ is a subflow of $(X_{\zeta},T_{\zeta})$ for all ordinals $0\leq\xi<\zeta\leq\eta$.
\item For every ordinal $0\leq\xi<\eta$ the flow $(X_{\xi+1},T_{\xi+1})$ is an isometric extension of $(X_\xi,T_\xi)$.
\item For a limit ordinal $0<\xi\leq\eta$ the flow $(X_\xi,T_\xi)$ is the inverse limit of the flows $\{(X_{\zeta},T_{\zeta}):0\leq\zeta<\xi\}$.
\end{enumerate}
A system $\{(X_\xi,T_\xi):0\leq\xi\leq\eta\}$ with the properties above is called a quasi-isometric system (I-system).
\end{fact}

\begin{defi}
A quasi-isometric system $\{(X_\xi,T_\xi):0\leq\xi\leq\eta\}$ is called \emph{normal}, if $(X_{\xi+1},T_{\xi+1})$ is the maximal isometric extension of $(X_{\xi},T_\xi)$ in $(X_{\eta},T_{\eta})$ for each ordinal $0\leq\xi\leq\eta$.
This quasi-isometric system is unique and it gives the minimal ordinal $\eta$ to represent the compact metric flow $(X,T)=(X_\eta,T_\eta)$.
(cf. \cite{Fu}, Proposition 13.1, Definition 13.2, Definition 13.3)
\end{defi}

The connectedness of fibres in isometric extensions will be essential in our arguments, and it will be ensured by representing the minimal compact metric flow by the \emph{normal} quasi-isometric system.

\begin{proposition}
Let $\{(X_\xi,T_\xi):0\leq\xi\leq\eta\}$ be a normal quasi-isometric system.
Then the flow $(X_1,T_1)$ is a minimal rotation on a compact metric group which is not necessarily connected, while for every ordinal $1\leq\xi<\eta$ the isometric extension from $(X_\xi,T_\xi)$ to $(X_{\xi+1},T_{\xi+1})$ has a connected fibre space.
\end{proposition}

\begin{proof}
We use the terminology and the results out of the paper \cite{MMWu}.
For an ordinal $1\leq\xi<\eta$, which is not a limit ordinal, we let $S(\pi_{\xi-1}^{\xi+1})$ be the relativised equicontinuous structure relation of the factor map $\pi_{\xi-1}^{\xi+1}$.
The compact metric flow $(X_{\xi+1},T_{\xi+1})/S(\pi_{\xi-1}^{\xi+1})$ is the maximal isometric extension of $(X_{\xi-1},T_{\xi-1})$ in $(X_{\xi+1},T_{\xi+1})$, which coincides with the maximal isometric extension $(X_\xi,T_\xi)$ of $(X_{\xi-1},T_{\xi-1})$ in $(X,T)$.
Thus by Theorem 3.7 of \cite{MMWu} the factor map of $(X_{\xi+1},T_{\xi+1})$ onto $(X_\xi,T_\xi)$ has connected fibres.

The same argument shows the connectedness of the fibres of the factor map $\pi_\zeta^{\gamma+1}$ in the case of a limit ordinal $1<\gamma<\eta$ and an ordinal $0\leq\zeta<\gamma$.
For every $x_\gamma\in X_\gamma$ we have that
\begin{equation*}
(\pi_\gamma^{\gamma+1})^{-1}(x_\gamma)=\bigcap_{0\leq\zeta<\gamma}(\pi_{\zeta+1}^{\gamma+1})^{-1}(\pi_{\zeta+1}^\gamma(x_\gamma)) ,
\end{equation*}
and therefore the fibre $(\pi_\gamma^{\gamma+1})^{-1}(x_\gamma)$ is the limit of a sequence of connected sets in a compact metric space, which is connected by \cite{K}, p.170, Theorem 14.
\end{proof}

We shall henceforth assume that $\{(X_\xi,T_\xi):0\leq\xi\leq\eta\}$ is the normal quasi-isometric system with $(X_\eta,T_\eta)=(X,T)$.
Moreover, for every ordinal $1\leq\xi<\eta$ we want to define a projection $f_\xi:X_\xi\longrightarrow\R$ of the function $f:X\longrightarrow\R$, which is associated to the factor map $\pi_\xi:(X,T)\longrightarrow (X_\xi,T_\xi)$.
These projections can be defined by families of probability measures, using the fact that every distal extension of compact metric flows is a so-called \emph{RIM}-extension (relatively invariant measure, cf. \cite{Gl}).
For an isometric extension this relatively invariant measure is unique (cf. \cite{Gl}), and we shall choose the most canonical family of measures for the distal extensions in the quasi-isometric system.
These families of measures obey to an integral decomposition formula within the quasi-isometric system.

\begin{proposition}\label{prop:meas}
For every ordinal $0\leq\xi\leq\eta$ there exists a family of probability measures $\{\mu_{\xi,x}:x\in X_\xi\}$ on $X$ so that for every $x\in X_\xi$ it holds true that
\begin{equation*}
\mu_{\xi,x}(\pi_\xi^{-1}(x))=1\enspace\textup{and}\enspace\mu_{\xi,x}\circ T^{-1}=\mu_{\xi, T_\xi x}.
\end{equation*}
Moreover, the mapping $x\mapsto\mu_{\xi,x}$ is continuous with respect to the weak-* topology on $C(X)^*$, and for a continuous function $\phi\in C(X)$ and ordinals $1\leq\xi<\zeta\leq\eta$ we have the equality that
\begin{equation}\label{eq:dec}
\mu_{\xi,x_\xi}(\phi)=\int_{X_\zeta}\mu_{\zeta,y_\zeta}(\phi)\, d(\mu_{\xi,x_\xi}\circ\pi_\zeta^{-1})(y_\zeta)
\end{equation}
for all $x_\xi\in X_\xi$.
\end{proposition}

\begin{proof}
The proof follows the construction of an invariant measure for $(X,T)$ in Chapter 12 of \cite{Fu}.
For every ordinal $\zeta$ with $\xi\leq\zeta\leq\eta$ we shall construct a probability measure $\mu^\zeta_{\xi,x}$ on $X_\zeta$ with the required properties, but with $(X_\zeta,T_\zeta)$ replacing $(X,T)$.
For $\zeta=\xi$ we put $\mu^\xi_{\xi,x}=\delta_x$, the measure with mass one on the point $x\in X_\xi$.
Suppose that for an ordinal $\zeta$ with $\xi\leq\zeta<\eta$ there exists a suitable probability measure $\mu^\zeta_{\xi,x}$ on $X_\zeta$.
For every $y\in X_\zeta$ we let $\mu_{y,\zeta}^{\zeta+1}$ be the unique measure on the fibre $(\pi_\zeta^{\zeta+1})^{-1}(y)$, which is invariant under all isometries of the fibre.
In the proof of Proposition 12.1 on \cite{Fu} it is verified that the mapping $y\mapsto\mu_{y,\zeta}^{\zeta+1}$ is continuous with respect to the weak-* topology on $C(X_{\zeta+1})^*$ and that $\mu_{y,\zeta}^{\zeta+1}\circ T_{\zeta+1}^{-1}=\mu_{T_\zeta y,\zeta}^{\zeta+1}$.
Now we define a RIM for the extension $(X_{\zeta+1},T_{\zeta+1})$ of $(X_\xi,T_\xi)$ by putting
\begin{equation}\label{eq:ext}
\mu^{\zeta+1}_{\xi,x}(\phi)=\int_{X_\zeta}\mu_{\zeta,y}^{\zeta+1}(\phi)\, d\mu_{\xi,x}^\zeta (y)
\end{equation}
for every $\phi\in C(X_{\zeta+1})$.
It is easily verified that this measure meets the requirements.

Furthermore, given a limit ordinal $\gamma$ with $1\leq\xi<\gamma\leq\eta$ and RIM's $\mu^\zeta_{\xi,x}$ on $X_\zeta$ for all ordinals $\zeta$ with $\xi\leq\zeta<\gamma$, we need to prove that there exists also a RIM $\mu^\gamma_{\xi,x}$ on $X_\gamma$ with the required properties.
For an ordinal $\xi\leq\zeta<\gamma$ and $x\in X_\xi$ the measure $\mu^\zeta_{\xi,x}$ defines a linear functional on the subspace of $C(X_\gamma)$ given by functions of the form $\phi\circ\pi_\zeta^\gamma$ with $\phi\in C(X_\zeta)$.
This functional can be extended to a functional on $C(X_\gamma)$ without increasing its norm, giving rise to a probability measure on $X_\gamma$.
By the continuity of the map $x\mapsto\mu^\zeta_{\xi,x}$ the set
\begin{equation*}
K_\zeta=\{(x,\nu):\nu\circ(\pi_\zeta^\gamma)^{-1}=\mu^\zeta_{\xi,x}\}\subset X_\xi\times C(X_\gamma)^*
\end{equation*}
is closed and compact in the product topology of $X_\xi$ and the weak-* topology on $C(X_\gamma)^*$.
Therefore also the set $K=\cap_{\xi\leq\zeta<\gamma}K_\zeta$ is compact, and by the finite intersection property every section $K^x=\{\nu\in C(X_\gamma)^*:(x,\nu)\in K\}$ with $x\in X_\gamma$ is non-empty.
Furthermore, two distinguished elements $\nu_i\in K^x$, $i\in\{1,2\}$ can be distinguished by a continuous function on $X_\gamma$, and for every large enough ordinal $\zeta<\gamma$ as well by a continuous function of the form $\phi\circ\pi_\zeta^\gamma$ with $\phi\in C(X_\zeta)$.
This contradicts however that $\nu_i\circ(\pi_\zeta^\gamma)^{-1}=\mu^\zeta_{\xi,x}$, and the section $K^x$ is a singleton for every $x\in X_\gamma$.
Thus the set $K\subset X_\xi\times C(X_\gamma)^*$ is the closed graph of the continuous function $x\mapsto\mu^\gamma_{\xi,x}$ on $X_\xi$, and the assertion that $\mu^\gamma_{\xi,x}\circ T_\gamma^{-1}=\mu^\gamma_{\xi, T_\xi x}$ can be verified by the same approximation argument.

The existence of the RIM $\{\mu_{\xi,x}:x\in X_\xi\}$ follows now by transfinite induction, and it remains to prove equality (\ref{eq:dec}).
The condition that
\begin{equation}\label{eq:dec_2}
\mu^{\alpha}_{\xi,x_\xi}(\phi)=\int_{X_\zeta}\mu_{\zeta,y_\zeta}^{\alpha}(\phi)\, d\mu_{\xi,x_\xi}^\zeta (y_\zeta)
\end{equation}
for every $\phi\in C(X_\alpha)$ holds true for the ordinal $\alpha=\zeta+1$ by the definition (\ref{eq:ext}).
If the equality (\ref{eq:dec_2}) holds true for an ordinal $\alpha>\zeta$, then the equality (\ref{eq:ext}) implies that
\begin{eqnarray*}
\mu^{\alpha+1}_{\xi,x_\xi}(\phi) = \int_{X_\alpha}\mu_{\alpha,y_\alpha}^{\alpha+1}(\phi)\, d\mu_{\xi,x_\xi}^\alpha (y_\alpha)=\hspace{5.3cm}\\ = \int_{X_\zeta}\left(\int_{X_\alpha}\mu_{\alpha,z_\alpha}^{\alpha+1}(\phi)\, d\mu_{\zeta,y_\zeta}^{\alpha}(z_\alpha)\right)\, d\mu_{\xi,x_\xi}^\zeta (y_\zeta) = \int_{X_\zeta} \mu_{\zeta,y_\zeta}^{\alpha+1}(\phi)\, d\mu_{\xi,x_\xi}^\zeta (y_\zeta)
\end{eqnarray*}
holds as well for every $\phi\in C(X_{\alpha+1})$.
Moreover, we can extend the equality (\ref{eq:dec_2}) with the approximation argument above also to a limit ordinal $\gamma$ with $\zeta<\gamma\leq\eta$, and the equality (\ref{eq:dec}) follows now by transfinite induction.
\end{proof}

Now we can define a continuous function $f_\xi:X_\xi\longrightarrow\R$ for every ordinal $\xi$ with  $1\leq\xi<\eta$ by
\begin{equation*}
f_\xi(x_\xi)=\mu_{\xi,x_\xi}(f) ,
\end{equation*}
and we can compute for ordinals $\xi$, $\zeta$ with $1\leq\xi<\zeta\leq\eta$ and an integer $n\neq 0$ that
\begin{eqnarray*}
(f_\zeta-f_\xi\circ\pi_\xi^\zeta)(n,x_\zeta) & = & \sum_{k=0}^{n-1}\left(\mu_{\zeta,T_\zeta^k x_\zeta}(f)-\mu_{\xi,\pi_\xi^\zeta(T_\zeta^k x_\zeta)}(f)\right)\\
& = & \sum_{k=0}^{n-1}\left(\mu_{\zeta,x_\zeta}(f\circ T^k)-\mu_{\xi,\pi_\xi^\zeta(x_\zeta)} (f\circ T^k)\right)\\
& = & \mu_{\zeta,x_\zeta}(f(n,\cdot))- \mu_{\xi,\pi_\xi^\zeta(x_\zeta)} (f(n,\cdot))
\end{eqnarray*}
holds true for every $x_\zeta\in X_\zeta$.
Hence the integral by the measure $d(\mu_{\xi,x_\xi}\circ\pi_\zeta^{-1})$, which is supported by the fibre $(\pi_\xi^\zeta)^{-1}(x_\xi)\subset X_\zeta$, turns out to be zero for every $x_\xi\in X_\xi$:
\begin{eqnarray*}
\int_{X_\xi}(f_\zeta-f_\xi\circ\pi_\xi^\zeta)(n,x_\zeta)\, d(\mu_{\xi,x_\xi}\circ\pi_\zeta^{-1})(x_\zeta)=\hspace{3cm}\\
=\int_{X_\xi}\left( \mu_{\zeta,x_\zeta}(f(n,\cdot))\right)\, d(\mu_{\xi,x_\xi}\circ\pi_\zeta^{-1})(x_\zeta)- \mu_{\xi,\pi_\xi^\zeta(x_\zeta)} (f(n,\cdot))=0
\end{eqnarray*}
The connectedness of the $\pi_\xi^\zeta$-fibres for $1\leq\xi<\zeta\leq\eta$ implies now that the function $(f_\zeta-f_\xi\circ\pi_\xi^\zeta)(n,x_\zeta)$ has a zero in the fibre $(\pi_\xi^\zeta)^{-1}(x_\xi)$ for every integer $n$ and every $x_\xi\in X_\xi$.
This property will be essential in the proofs of the Lemmas \ref{lem:c_t} and \ref{lem:lim}.
Furthermore, the following representation of isometric extensions in terms of compact metric group extensions will be essential:

\begin{fact}\label{fact:iso}
A minimal isometric extension $(Z,R)$ of a minimal compact metric flow $(Y,S)=\sigma(Z,R)$ can be represented by a minimal isometric group extension $(\tilde Z,\tilde R)$ of $(Y,S)=\tilde \sigma(\tilde Z,\tilde R)$, with a compact metric group $K\subset\textup{Aut}(\tilde Z,\tilde R)$ acting freely on the fibres $\tilde\sigma^{-1}(\tilde\sigma(\tilde z))=\{g(\tilde z):g\in K\}$ for every $\tilde z\in\tilde Z$, and then taking the orbit space by a closed subgroup $H$ of $K$ (cf. chapter 5 in \cite{GlWe}).
\end{fact}

The ``local'' behaviour of an isometric \emph{group} extension is similar to a skew product extension by a compact metric group, even if the global structure might be different in the sense that the space does not split into a product.

\begin{lemma}\label{lem:iso}
Let $(\tilde Z,\tilde R)$ be a minimal isometric \emph{group} extension of $(Y,S)=\tilde \sigma(\tilde Z,\tilde R)$.
Then for every $\varepsilon>0$ there exists a $\delta>0$ so that $d_{\tilde Z} (\tilde x,\tilde R^n \tilde x)<\delta$ for $\tilde x\in\tilde Z$ and an integer $n$ implies that $d_{\tilde Z} (\tilde y, \tilde R^n\tilde y)<\varepsilon$ for every $\tilde y\in\tilde Z$ with $\tilde\sigma(\tilde y)=\tilde\sigma(\tilde x)$.
\end{lemma}

\begin{proof}
A compact metric group $K\subset\textup{Aut}(\tilde Z,\tilde R)$ acting freely on the fibres defines a uniformly equicontinuous set of homeomorphisms of $\tilde Z$.
Thus there exists a $\delta>0$ so that for all $\tilde x,\tilde y\in\tilde Z$ with $d_{\tilde Z}(\tilde x,\tilde y))<\delta$ and all $g\in K$ it holds true that $d_{\tilde Z}(g(\tilde x),g(\tilde y))<\varepsilon$.
For a point $\tilde x\in\tilde Z$ and an integer $n$ with $d_{\tilde Z} (\tilde x,\tilde R^n \tilde x)<\delta$ it follows then for all $g\in K$ that $d_{\tilde Z}(g(\tilde x),g(\tilde R^n\tilde x))=d_{\tilde Z}(g(\tilde x),\tilde R^n g(\tilde x))<\varepsilon$, and as the $K$-orbit of $\tilde x$ is all of $\tilde\sigma^{-1}(\tilde\sigma(\tilde x))$ the lemma is verified.
\end{proof}

We want to use these tools to study how the dynamical properties of the skew product extensions $\mathbf T_{\xi, f_\xi}:X_\xi\times\R\longrightarrow X_\xi\times\R$ change over the ordinals $0\leq\xi\leq\eta$.
At first we want to consider the step from an ordinal to its successor.

\begin{lemma}\label{lem:c_t}
Let $\gamma$ be an ordinal with $1\leq\gamma<\eta$.
Then the cocycle 
\begin{equation*}
(f_{\gamma+1}-f_\gamma\circ\pi_\gamma^{\gamma+1})(n,x_{\gamma+1})
\end{equation*}
is either a coboundary or it has a topologically transitive skew product.
Therefore if $f_\gamma(n,x_\gamma)$ is a coboundary, then $f_{\gamma+1}(n,x_{\gamma+1})$ is either a coboundary or it has a topologically transitive skew product.
Furthermore, if $f_\gamma(n,x_\gamma)$ is transient, then $f_{\gamma+1}(n,x_{\gamma+1})$ is either transient or it has a topologically transitive skew product.
\end{lemma}

\begin{proof}
Let $K\subset\textup{Aut}(\tilde X,\tilde T)$ be a compact metric group extension of $(X_\gamma, T_\gamma)$ with a compact subgroup $H\subset K$ so that $(X_{\gamma+1},T_{\gamma+1})=\tau(\tilde X,\tilde T)$ is the $H$-orbit space in $(\tilde X,\tilde T)$.
If the skew product extension of $h(n,x_{\gamma+1})=(f_{\gamma+1}-f_\gamma\circ\pi_\gamma^{\gamma+1})(n,x_{\gamma+1})$ is not topologically transitive, then transitivity is also not valid for the skew product $\mathbf{\tilde T}_{h\circ\tau}$.
By Lemma \ref{lem:at} there exist thus a compact symmetric neighbourhood $L\subset\R$ of zero and an $\varepsilon>0$, so that $\tilde d(\tilde x, \tilde T^n \tilde x)<\varepsilon$ for $\tilde x\in\tilde X$ and $n\in\Z$ implies that $(h\circ\tau)(n,\tilde x)\notin 2L\setminus L^0$.
Moreover, by Lemma \ref{lem:iso} there exists a $\delta>0$, so that $\tilde d(\tilde x,\tilde T^n \tilde x)<\delta$ for some $\tilde x\in\tilde X$ and $n\in\Z$ is sufficient for $\tilde d(\tilde y,\tilde T^n\tilde y)<\varepsilon$ for every $\tilde y\in\tilde X$ with $\pi_\gamma^{\gamma+1}\circ\tau(\tilde y)=\pi_\gamma^{\gamma+1}\circ\tau(\tilde x)$.
Now let $\{(n_k,\tilde x_k)\}_{k\geq 1}\subset\Z\times \tilde X$ be a sequence with $\tilde d(\tilde x_k, \tilde T^{n_k} \tilde x_k)\to 0$.
The cocycle $(f_{\gamma+1}-f_\gamma\circ\pi_\gamma^{\gamma+1})(n_k,y)$ has a zero $y_k$ and a connected range on the fibre $(\pi_\gamma^{\gamma+1})^{-1}(\pi_\gamma^{\gamma+1}(\tau(\tilde x_k)))$, and thus for all $k\geq 1$ with $\tilde d(\tilde x_k,\tilde T^{n_k}\tilde x_k)<\delta$ the assertion that $(h\circ\tau)(n_k,\tilde x_k)\notin 2L\setminus L^0$ implies that $(h\circ\tau)(n_k,\tilde x_k)\in L$.
This argument can be repeated with an arbitrarily small neighbourhood $L$, and therefore we have the convergence that $(h\circ\tau)(n_k,\tilde x_k)\to 0$ as $k\to\infty$. 
By Proposition 3.4 in \cite{LM} the cocycle $(h\circ\tau)(n,\tilde x)$ is a coboundary, and from the uniform boundedness of $(h\circ\tau)(n,\tilde x)$ follows the uniform boundedness of the cocycle $h(n,x_{\gamma+1})$, which is hence also a coboundary.

Now suppose that $f_\gamma(n,x_\gamma)$ is transient, while $f_{\gamma+1}(n,x_{\gamma+1})$ is recurrent and its skew product is not transitive.
Then the skew product $(f_{\gamma+1}\circ\tau)(n,\tilde x)$ is as well not transitive, and again we choose a compact symmetric neighbourhood $L\subset\R$ of zero and an $\varepsilon>0$ (cf. Lemma \ref{lem:at}) so that $\tilde d(\tilde x,\tilde T^n\tilde x)<\varepsilon$ for $\tilde x\in\tilde X$ and $n\in\Z$ implies that $(f_{\gamma+1}\circ\tau)(n,\tilde x)\notin 2L\setminus L^0$.
By Lemma \ref{lem:iso} there exists a $\delta>0$ so that $\tilde d(\tilde x,\tilde T^n\tilde x)<\delta$ for some $\tilde x\in\tilde X$ and $n\in\Z$ is sufficient for $\tilde d(\tilde y,\tilde T^n\tilde y)<\varepsilon$ for every $\tilde y\in\tilde X$ with $\pi_\gamma^{\gamma+1}\circ\tau(\tilde y)=\pi_\gamma^{\gamma+1}\circ\tau(\tilde x)$.
If $\{(n_k,\tilde x_k)\}_{k\geq 1}\subset\Z\times\tilde X$ is a sequence with $\tilde d(\tilde x_k,\tilde T^{n_k}\tilde x_k)\to 0$, then the transience of $f_\gamma(n,x_\gamma)$ implies that $(f_\gamma\circ\pi_\gamma^{\gamma+1}\circ\tau)(n_k,\tilde x_k)\notin 2L$ for all large enough integers $k$ (cf. Remark \ref{rems:rec}).
The cocycle $(f_{\gamma+1}-f_\gamma\circ\pi_\gamma^{\gamma+1})(n_k,y)$ has a zero $y_k\in(\pi_\gamma^{\gamma+1})^{-1}(\pi_\gamma^{\gamma+1}\circ\tau(\tilde x_k))$, and for $\tilde y_k\in\tau^{-1}(y_k)$ it holds true that $(f_{\gamma+1}\circ\tau)(n_k,\tilde y_k)\notin 2L$.
We obtain from the choice of $\delta$ and the connectedness of the $(f_{\gamma+1}\circ\tau)(n_k,\tilde x)$-range on the $(\pi_\gamma^{\gamma+1}\circ\tau)$-fibre that $(f_{\gamma+1}\circ\tau)(n_k,\tilde x_k)\notin 2L$ for all large enough integers $k$.
However, given a point $x_{\gamma+1}\in X_{\gamma+1}$ and a sequence $\{m_k\}_{k\geq 1}$ of integers with $\mathbf T_{\gamma+1,f_{\gamma+1}}^{m_k}(x_{\gamma+1},0)\to(x_{\gamma+1},0)$, we can choose a point $\tilde x\in\tau^{-1}(x_{\gamma+1})$ and a subsequence $\{m_{k_l}\}_{l\geq 1}\subset\Z$ so that $\{\tilde T^{m_{k_l}}\tilde x\}_{l\geq 1}$ is convergent in $\tilde X$.
Now a contradiction occurs for the sequence $\{(n_k,\tilde x_k)=(m_{k_{l+1}}-m_{k_l},\tilde T^{m_{k_l}}\tilde x)\}_{k\geq 1}\subset\Z\times\tilde X$.
\end{proof}

The arguments are somehow similar in the case of a limit ordinal, but instead of Lemma \ref{lem:iso} an approximation of the limit ordinal will be applied.

\begin{lemma}\label{lem:lim}
Let $\gamma$ be a limit ordinal with $1<\gamma\leq\eta$.
If $f_\xi(n,x_\xi)$ is a coboundary for all $1\leq\xi<\gamma$, then $f_\gamma(n,x_\gamma)$ is either a coboundary or it has a topologically transitive skew product extension.
If there exists an ordinal $1\leq\zeta<\gamma$ so that $f_\xi(n,x_\xi)$ is transient for all $\zeta\leq\xi<\gamma$, then $f_\gamma(n,x_\gamma)$ is either transient or it has a topologically transitive skew product extension.
Furthermore, if for every ordinal $1<\zeta<\gamma$ there exists an ordinal $\zeta\leq\xi<\gamma$ so that $f_\xi(n,x_\xi)$ has a topologically transitive skew product extension, then $f_\gamma(n,x_\gamma)$ has a topologically transitive skew product extension.
\end{lemma}

\begin{proof}
Suppose that the skew product of $\mathbf T_{\gamma,f_\gamma}$ is not transitive and $f_\xi(n,x_\xi)$ is a coboundary for every $1\leq\xi<\gamma$, and let $\{(n_k,x_k)\}_{k\geq 1}\subset\Z\times X_\gamma$ be a sequence with $d_\gamma (x_k, T_\gamma^{n_k} x_k)\to 0$.
By Lemma \ref{lem:at} there exist a compact symmetric neighbourhood $L\subset\R$ of zero and an $\varepsilon>0$ so that $d_\gamma(x, T_\gamma ^n x)<\varepsilon$ for $x\in X_\gamma$ and $n\in\Z$ implies that $f_\gamma(n,x)\notin 2L\setminus L^0$.
As $\gamma$ is a limit ordinal, we can choose an ordinal $\zeta<\gamma$ so that $d_\gamma (x,y)<\varepsilon/3$ holds true for all $x,y\in X_\gamma$ with $\pi_\zeta^\gamma (x)=\pi_\zeta^\gamma (y)$.
For all positive integers $k$ with $d_\gamma(x_k, T_\gamma^{n_k} x_k)<\varepsilon/3$ and all $y_k\in X_\gamma$ with $\pi_\zeta^\gamma (y_k)=\pi_\zeta^\gamma (x_k)$ it follows that $d_\gamma(y_k, T_\gamma^{n_k} y_k)<\varepsilon$.
Moreover, for all large enough integers $k$ it holds true that $f_\zeta (n_k,\pi_\zeta^\gamma(x_k))\in L$, because $f_\zeta$ is a coboundary.
The cocycle $(f_\gamma-f_\zeta\circ\pi_\zeta^\gamma)(n_k,x)$ has a zero $y_k\in(\pi_\zeta^\gamma)^{-1}(\pi_\zeta^\gamma(x_k))$ and thus $f_\gamma(n_k,y_k)\in L$ for every $k\geq k_0$ and a suitable integer $k_0\geq 1$.
We can now conclude from the connectedness of the fibre and $f_\gamma(n_k,x_k)\notin 2L\setminus L^0$ that $f_\gamma (n_k,x_k)\in L$ holds true for all $k\geq k_0$.
The neighbourhood $L$ can be chosen arbitrarily small, and therefore $f_\gamma(n_k,x_k)\to 0$ as $k\to\infty$.
As the sequence $\{(n_k,x_k)\}_{k\geq 1}\subset\Z\times X_\gamma$ was arbitrary, the Proposition 3.4 in \cite{LM} asserts that $f_\gamma(n,x)$ is a coboundary.

Now suppose that there exists an ordinal $1\leq\zeta<\gamma$ so that $f_\xi(n,x_\xi)$ is transient for all $\zeta\leq\xi<\gamma$, while $f_\gamma(n,x)$ is recurrent but its skew product extension is not topologically transitive.
Let $\{(n_k,x_k)\}_{k\geq 1}\subset\Z\times X_\gamma$ be a sequence with $d_\gamma (x_k, T_\gamma^{n_k} x_k)\to 0$ and $f_\gamma(n_k,x_k)\to 0$, and choose as above compact symmetric neighbourhood $L\subset\R$ of zero, an $\varepsilon>0$, and an ordinal $\zeta<\gamma$.
For all large enough integers $k$ it holds true that $d_\gamma(x_k, T_\gamma^{n_k} x_k)<\varepsilon/3$, and thus $d_\gamma(y_k, T_\gamma^{n_k} y_k)<\varepsilon$ for all $y_k$ with $\pi_\zeta^\gamma (y_k)=\pi_\zeta^\gamma (x_k)$.
Furthermore, for all large enough integers $k$ it follows that $f_\zeta (n_k,\pi_\zeta^\gamma(x_k))\notin 2L$, because $f_\zeta(n,x)$ is transient.
The cocycle $(f_\gamma-f_\zeta\circ\pi_\zeta^\gamma)(n_k,x)$ has a zero $y_k$ on the fibre $(\pi_\zeta^\gamma)^{-1}(\pi_\zeta^\gamma(x_k))$ for every integer $k\geq 1$, and thus $f_\gamma(n_k,y_k)\notin 2L$ for all large enough integers $k$.
Now the connectedness of the fibre $(\pi_\zeta^\gamma)^{-1}(\pi_\zeta^\gamma(x_k))$ and $f_\gamma(n_k,x_k)\notin 2L\setminus L^0$ imply for all large enough integers $k$ that $f_\gamma(n_k,x_k)\notin 2L$, in contradiction to $f_\gamma(n_k,x_k)\to 0$.

For the proof of the last assertion suppose that the skew product of $f_\gamma(n,x_\gamma)$ is not topologically transitive, and then choose as above a compact symmetric neighbourhood $L\subset\R$ of zero, an $\varepsilon>0$, and an ordinal $\zeta<\gamma$.
Now use the transitivity of the skew product of $f_\xi(n,x_\xi)$ for a suitable ordinal $\xi$, $\zeta\leq\xi<\gamma$ and the fact that $(f_\gamma-f_\xi\circ\pi_\xi^\gamma)(n,x)$ has a zero on every $\pi_\xi^\gamma$-fibre.
Then a contradiction occurs to $f_\gamma(n,x)\in 2L\setminus L^0$ for all $(n,x)\in\Z\times X_\gamma$ with $d_\gamma(x, T_\gamma^n x)<\varepsilon$.
\end{proof}

\begin{proposition}\label{prop:max}
If the real-valued cocycle $f(n,x)$ is topologically recurrent, then either there exists a maximal ordinal $1\leq\alpha\leq\eta$ so that the skew product extension $\mathbf T_{\alpha,f_\alpha}$ is topologically transitive or $f(n,x)$ is a coboundary.
\end{proposition}

\begin{proof}
We suppose at first that the cocycle $f_\xi(n,x_\xi)$ is recurrent for every ordinal $1\leq\xi<\eta$.
The cocycle $f_1(n,x_1)$ is defined over a minimal rotation on a compact metric group, and by Theorem 1 in \cite{LM} either $\mathbf T_{1,f_1}$ is transitive or $f_1$ is a coboundary.
In any case, the Lemmas \ref{lem:c_t}, \ref{lem:lim}, and transfinite induction imply that either there exists an ordinal $1\leq\zeta\leq\eta$ so that $\mathbf T_{\zeta,f_\zeta}$ is topologically transitive or $f(n,x)$ is a coboundary.

If $f_\xi(n,x_\xi)$ is transient for an ordinal $1\leq\xi<\eta$, then let $\beta$ be the minimal element of the set of ordinals $\xi<\zeta\leq\eta$ so that $f_\zeta(n,x_\zeta)$ is topologically recurrent.
This set is of course non-empty, because $f_\eta(n,x_\eta)$ is topologically recurrent, and it follows from the Lemmas \ref{lem:c_t} and \ref{lem:lim} that $\mathbf T_{\beta,f_\beta}$ is even topologically transitive.

Now we consider the set of ordinals $1\leq\zeta\leq\eta$ so that for all $\zeta\leq\xi\leq\eta$ the skew product $\mathbf T_{\xi,f_\xi}$ is \emph{not} topologically transitive.
If this set is empty, then $\mathbf T_{\eta,f_\eta}$ is topologically transitive and $\alpha=\eta$.
Otherwise, there exists a minimal element, which cannot be a limit ordinal by Lemma \ref{lem:lim}, and thus we have a maximal ordinal $1\leq\alpha\leq\eta$ so that $\mathbf T_{\alpha,f_\alpha}$ is topologically transitive.
\end{proof}

It should be mentioned that in the case of an uniquely ergodic homeomorphism the proofs above could be simplified.
Then the recurrence of the cocycle $f(n,x)$ implies that $\mu_0(f)=0$, because $\mu_0$ is the unique invariant probability measure of full support and $f(n,x)/n$ converges uniformly to $\mu_0(f)$ as $n\to\infty$.
The homeomorphisms $T_\xi$ for the ordinals $0\leq\xi\leq\eta$ are also uniquely ergodic, and from $\mu_\xi(f_\xi)=0$ follows for every ordinal $1\leq\xi\leq\eta$ the measure theoretic (and therefore topological) recurrence of the cocycle $f_\xi(n,x_\xi)$.
However, for a non-uniquely ergodic homeomorphism $T$ there might be ordinals $0\leq\xi<\zeta\leq\eta$ so that $f_\xi(n,x_\xi)$ is transient, while the increment to $f_\zeta(n,x_\zeta)$ forces the topological transitivity of the skew product $\mathbf T_{\zeta,f_\zeta}$.

After the flow $(X_\alpha,T_\alpha)$ with a topologically transitive skew product has been identified, the extension from $(X_\alpha,T_\alpha)$ to $(X,T)$ will become the object of study.
There might be infinitely many isometric extensions in between, and therefore this extension is in general only a distal extension.
For distal extensions there is a result similar to Fact \ref{fact:iso}, however with a Hausdorff topological group acting on a compact Hausdorff space and both of them in general not being metrisable.
This is a result of Ellis (cf. 12.12, 12.13, and 14.26 of \cite{El}), while a direct and simple proof is given in Proposition 1.1 of \cite{MMWu}.

\begin{fact}\label{fact:ex}
Let $(Z,R)=\tau(X,T)$ be a factor of a distal minimal compact Hausdorff flow $(X,T)$.
Then there exists a distal minimal compact Hausdorff flow $(X',T')$ with $(X,T)=\pi(X',T')$ as a factor and a Hausdorff topological group $G$ acting transitively (in the strict sense) and freely on the fibres of the factor map $\tau\circ\pi:(X',T')\longrightarrow (Z,R)$ by automorphisms of $(X',T')$.
Moreover, there exists a subgroup $L$ of $G$ so that the factor map $\pi:(X',T')\longrightarrow (X,T)$ is the mapping of a point $x\in X'$ onto its $L$-orbit.
\end{fact}

In the paper \cite{Gl2} it is proven that the metrisability of a compact Hausdorff space $(X',T')$ with these properties implies even that the extension from $(Z,R)$ to $(X,T)$ is an isometric extension.

\begin{proposition}\label{prop:flow}
There exists a factor $(Y,S)=(X_\alpha\times M,T_{\alpha,\Phi,f_\alpha})=\pi_Y(X,T)$, which is a Rokhlin extension of $(X_\alpha,T_\alpha)=\tau_\alpha(Y,S)$ by a distal minimal $\R$-flow $(M,\{\Phi_t:t\in\R\})$ on a compact metric space $(M,\delta)$ and the function $f_\alpha:X_\alpha\longrightarrow\R$.
The $\R$-flow $\{\Psi_t:t\in\R\}\subset\textup{Aut}(Y,S)$ defined by $\Psi_t(x,m)=(x,\Phi_t(m))$ for every $(x,m)\in Y$ fulfils that
\begin{equation}\label{eq:o_flow}
\bar{\mc O}_{S, f_\alpha\circ\tau_\alpha}(y,0)\cap(\tau_\alpha^{-1}(\tau_\alpha(y))\times\{t\})\subset\{(\Psi_t(y),t)\}
\end{equation}
for every $y\in Y$ and every $t\in\R$.
If $(\tau_\alpha(y),0)$ is a transitive point for $\mathbf T_{\alpha,f_\alpha}$, then these two sets coincide for every $t\in\R$.
Moreover, for every $x$ in a residual subset of $X$ it holds true that
\begin{equation}\label{eq:o_X}
\pi_Y^{-1}(x)\times\{0\}\subset\bar{\mc O}_{T, f_\alpha\circ\pi_\alpha}(x,0) .
\end{equation}
\end{proposition}

\begin{proof}
We shall construct a factor $(Y,S)$ of $(X,T)$ and a flow $\{\Psi_t:t\in\R\}\subset\textup{Aut}(Y,S)$, and thereafter it will be shown that $(Y,S)$ can be represented as a Rokhlin extension of $(X_\alpha,T_\alpha)$.
Let $(X',T')$ be a minimal compact Hausdorff extension of $(X_\alpha,T_\alpha)$ with $(X,T)=\pi(X',T')$ as a factor and a Hausdorff group $G$ acting freely by automorphisms of $(X',T')$ on the fibres of $\pi_\alpha\circ\pi$, so that $(X,T)$ is the $L$-orbit space of $(X',T')$ under a closed subgroup $L\subset G$ (cf. Fact \ref{fact:ex}).
For an arbitrary point $z'\in X'$ and $t\in\R$ we define a closed subset of $G$ by
\begin{equation}\label{eq:G}
G_{z',t}=\{g\in G: (\pi(g(z')),t)\in\mc D_{T,f_\alpha\circ\pi_\alpha}(\pi(z'),0)\} .
\end{equation}
The mapping $\pi$ is open as a factor mapping of distal flows, and therefore for every $g\in G_{z',t}$ there are nets $\{z_k'\}_{k\in I}\subset X'$ and $\{n_k\}_{k\in I}\subset\Z$ so that $z_k'\to z'$, $T^{n_k}\pi(z_k')\to\pi(g(z'))$, and $f_\alpha(n_k,\pi_\alpha\circ\pi(z_k'))\to t$.
We can conclude for every fixed integer $m$  that
\begin{equation*}
T^{n_k}\pi(T'^m z_k')=T^{n_k+m}\pi(z_k')\to T^m\pi(g(z'))=\pi(T'^m g(z'))=\pi(g(T'^m z'))
\end{equation*}
and
\begin{eqnarray*}
f_\alpha(n_k,\pi_\alpha\circ\pi(T'^m z_k')) & = & f_\alpha(n_k,\pi_\alpha\circ\pi(z_k'))-f_\alpha(m,\pi_\alpha\circ\pi(z_k'))\\
& & +f_\alpha(m,\pi_\alpha\circ T^{n_k}\circ\pi(z_k'))\to t ,
\end{eqnarray*}
because the function $f_\alpha\circ\pi_\alpha$ is constant on the fibres of $\pi_\alpha$.
The density of the $T$-orbit of $z'$ implies for every $x'\in X'$ that
\begin{equation*}
(\pi(g(x')),t)\in\mc D_{T,f_\alpha\circ\pi_\alpha}(\pi(x'),0)
\end{equation*}
and thus $g\in G_{x',t}=G_{z',t}=G_t$.
It follows by symmetry that $G_{-t}=(G_t)^{-1}$.

Now we fix a point $\tilde x\in X$ with $\bar{\mc O}_{T_\alpha,f_\alpha}(\pi_\alpha(\tilde x),0)=X_\alpha\times\R$ and $\mc D_{T,f_\alpha\circ\pi_\alpha}(\tilde x,0)=\bar{\mc O}_{T,f_\alpha\circ\pi_\alpha}(\tilde x,0)$ (cf. Lemma \ref{lem:o_p}).
We observe that $G_t$ is non-empty for every $t\in\R$, because due to $\bar{\mc O}_{T, f_\alpha}(\pi_\alpha(\tilde x),0)=X_\alpha\times\R$ and the compactness of $X$ the set
\begin{equation*}
\bar{\mc O}_{T,f_\alpha\circ\pi_\alpha}(\tilde x,0)\cap\pi_\alpha^{-1}(\pi_\alpha(\tilde x))\times\{t\}
\end{equation*}
is non-empty.
For $t,t'\in\R$ and $g\in G_t$, $g'\in G_{t'}$ select $x', z'\in X'$ so that $\pi(x')=\tilde x$, $x'=g'(z')$.
It follows that
\begin{equation*}
(\tilde x,t')=(\pi(g'(z')),t')\in\mc D_{T,f_\alpha\circ\pi_\alpha}(\pi(z'),0) ,
\end{equation*}
and for $y'=g(x')$ it holds true that $(\pi(y'),t)\in\bar{\mc O}_{T,f_\alpha\circ\pi_\alpha}(\tilde x,0)$.
We can conclude from $(\pi(y'),t+t')\in\mc D_{T,f_\alpha\circ\pi_\alpha}(\pi(z'),0)$ that $gg'\in G_{t+t'}$, and thus $G_t G_{t'}\subset G_{t+t'}$.
Hence the closed set $G_0$ is a closed subgroup of the Hausdorff topological group
\begin{equation*}
\tilde G=\cup_{t\in\R}G_t ,
\end{equation*}
so that the set $G_t$ is a $G_0$-coset in $\tilde G$ for every $t\in\R$.
Moreover, we have that $G_0\subset G_t G_0 (G_t)^{-1}\subset G_0$, and thus $G_0\supset L$ is normal in $\tilde G$.
Therefore the mapping $t\mapsto G_t$ is a group homomorphism from $\R$ into $\tilde G/G_0$.
We fix an arbitrary $z'\in X'$ and observe that the pre-image of the closed set $\mc D_{T,f_\alpha\circ\pi_\alpha}(\pi(z'),0)$ under the mapping $(g,t)\mapsto (\pi(g(z')),t)$ is the closed set
\begin{equation*}
\{(t,g):t\in\R, g\in G_t\}\subset\R\times G ,
\end{equation*}
and hence the group homomorphism $t\mapsto G_t$ is continuous with respect to the quotient topology on $\tilde G/G_0$.

By the definition of $G_0$ the orbit space $(Y,S)$ of $G_0$ on $X'$ consists of closed sets, and it is an extension of $(X_\alpha,T_\alpha)=\tau_\alpha(Y,S)$ and a factor of $(X,T)$, because $L$ is a subgroup of $G_0$.
The mapping from $x'\in X'$ to $G_0(x')$ is a factor mapping of a distal flow and therefore open, and thus the mapping $x\mapsto\pi_Y(x)=G_0(\pi^{-1}(x))$ is continuous with respect to the Hausdorff metric on $(Y,S)$.
We define the $\R$-action $\{\Psi_t:t\in\R\}\subset\textup{Aut}(Y,S)$ by
\begin{equation*}
\Psi_t(y)=G_t ((\pi_Y\circ\pi)^{-1}(y))=G_t (\{x'\in X':G_0(x')=y\})
\end{equation*}
for every $y\in Y$ and $t\in\R$, and its continuity follows from the continuity of $t\mapsto G_t$ as the group $G$ carries the compact open topology of its action on $X'$.

We turn to the inclusion (\ref{eq:o_flow}).
Suppose that $(y_i,t)$ for $i\in\{1,2\}$ are both within the intersection $\bar{\mc O}_{S, f_\alpha\circ\tau_\alpha}(y,0)\cap\tau_\alpha^{-1}(\tau_\alpha(y))\times\{t\}$, and select $x\in\pi_Y^{-1}(y)$.
By the compactness of the space $X$ there exist points $x_i\in\pi_Y^{-1}(y_i)\subset\pi_\alpha^{-1}(\tau_\alpha(y))$ so that $(x_i,t)\in\bar{\mc O}_{T, f_\alpha\circ\pi_\alpha}(x,0)$, and therefore $(x_2,0)\in\mc D_{T,f_\alpha\circ\pi_\alpha}(x_1,0)$.
The definition (\ref{eq:G}) implies that $y_1=\pi_Y(x_1)=\pi_Y(x_2)=y_2=\Psi_t(y)$, and thus for every $y\in Y$ and $t\in\R$ there can be at most one point in the intersection
\begin{equation*}
\bar{\mc O}_{S, f_\alpha\circ\tau_\alpha}(y,0)\cap\tau_\alpha^{-1}(\tau_\alpha(y))\times\{t\} .
\end{equation*}
If there exists such a point, then the inclusion (\ref{eq:o_flow}) holds true, and if the point $(\tau_\alpha(y),0)$ is transitive under $\mathbf T_{\alpha,f_\alpha}$, then by the compactness of $Y$ the this intersection is non-empty for every $t\in\R$.
The inclusion (\ref{eq:o_X}) on a residual set of $x\in X$ follows directly from Lemma \ref{lem:o_p} and the definition of the subgroup $G_0$, which defines the factor map $\pi_Y$.

We want to verify that the distality of $(X,T)$ implies that the flow $((n,t),y)\mapsto\Psi_t(S^n y)$ on $Y$ is distal.
We construct the compact Hausdorff space $X'$ as an uncountable product of copies of $X$ (cf. the proof of \cite{MMWu}, Proposition 1.1).
The group $G$ is then a quotient of the subgroup of the Ellis group $\mc E(X,T)$, which preserves a chosen $\pi_\alpha$-fibre in $X$, divided by its subgroup preserving every element in that fibre.
The group $G$ is acting on each coordinate of the product space $X'$, and it is equipped with the compact-open topology of its action on $X'$.
Even if the Ellis group $\mc E(X,T)$ does not act continuously on $X$, it acts distally in the sense that $\inf_{g\in\mc E(X,T)}d(g(x),g(x'))>0$ for distinct $x, x'\in X$.
This is an immediate consequence of $(X,T)$ being distal and $\mc E(X,T)$ being the closure of $\{T^n:n\in\Z\}$ in $X^X$.
Therefore also the action of $\Z\times G$ on $X'$ is distal, as it is defined as coordinate-wise action of elements of $\mc E(X,T)$, and the flow $(Y,\Z\times\{\Psi_t:t\in\R\})$ is distal as a factor of $(X',\Z\times\tilde G)$.

We define a continuous mapping on $Y$ by
\begin{equation*}
R(y)=\Psi_{-(f_\alpha\circ\tau_\alpha)(y)}(S y) ,
\end{equation*}
and as the flow $\{\Psi_t:t\in\R\}\subset\textup{Aut}(Y,S)$ leaves the function $f_\alpha\circ\tau_\alpha$ invariant, the continuous mapping $y\mapsto S^{-1}(\Psi_{(f_\alpha\circ\tau_\alpha)(S^{-1}y)}(y))$ is the inverse of $R$.
Therefore $R$ is a homeomorphism of $Y$ and an extension of $(X_\alpha,T_\alpha)=\tau_\alpha(Y,R)$ so that $\{\Psi_t:t\in\R\}\subset\textup{Aut}(Y,R)$.
By iteration it follows that $R^n y=\Psi_{-(f_\alpha\circ\tau_\alpha)(n,y)}(S^n y)$, and the distality of the $\Z\times\R$-action on $Y$ implies that the extension of $(X_\alpha,T_\alpha)$ to $(Y,R)$ is a distal extension.
Hence $(Y,R)$ is a distal flow, and $Y$ decomposes into minimal $R$-orbit closures.
Every $R$-orbit closure $C\subset Y$ is thus the orbit closure $\bar{\mc O}_R(y)$ of point $y\in Y$ with $\mathbf T_{\alpha,f_\alpha}(\tau_\alpha(S^n y),0)=X_\alpha\times\R$ for every integer $n$, and for such a point it follows from $R^n y=\Psi_{-(f_\alpha\circ\tau_\alpha)(n,y)}(S^n y)$ and the inclusion (\ref{eq:o_flow}) that $\bar{\mc O}_R(y)\times\{0\}\subset\bar{\mc O}_{S, f_\alpha\circ\tau_\alpha}(y,0)$.
Now the argument used in the proof of the inclusion (\ref{eq:o_flow}) asserts that every $R$-orbit closure intersects every $\tau_\alpha$-fibre in a single point.

By Lemma \ref{lem:o_p} there exists a residual set $\mc G\subset Y$ of points with coincidence of the $\{\Psi_t:t\in\R\}$-orbit and the $\{\Psi_t:t\in\R\}$-prolongation and coincidence of the $R$-orbit and the $R$-prolongation.
The set $\mc F=\tau_\alpha(\mc G)$ is also residual, because $\tau_\alpha$ is an open mapping, and for an arbitrary point $\tilde x\in\mc F$ and fixed point $\tilde y\in\mc G$ with $\tau_\alpha(\tilde y)=\tilde x$ the minimality of $(Y,S)$ and $S \tilde y=\Psi_{(f_\alpha\circ\tau_\alpha)(\tilde y)}(R \tilde y)$ imply that
\begin{equation*}
\tau_\alpha^{-1}(\tilde x)\subset\overline{\{\Psi_t(R^n \tilde y):(n,t)\in\Z\times\R\}} .
\end{equation*}
Now we can conclude from $\bar{\mc O}_R(\tilde y)\cap\tau_\alpha^{-1}(\tilde x)=\{\tilde y\}$ and the invariance of the $\tau_\alpha$-fibres under $\{\Psi_t:t\in\R\}$ that $\tau_\alpha^{-1}(\tilde x)$ is a subset of the $\{\Psi_t:t\in\R\}$-prolongation of $\tilde y$, which coincides with the $\{\Psi_t:t\in\R\}$-orbit closure of $\tilde y$.
Thus the orbit $\{\Psi_t(\tilde y):t\in\R\}$ is dense in $\tau_\alpha^{-1}(\tilde x)$, and therefore the distal $\{\Psi_t:t\in\R\}$-action is minimal on $\tau_\alpha^{-1}(\tilde x)$.
Moreover, we want to verify that $\bar{\mc O}_R(y)=\mc D_R(y)$ holds true for every $y\in\tau_\alpha^{-1}(\tilde x)$.
Otherwise, there exist distinct points $y'\in\bar{\mc O}_R(y)$ and $y''\in\mc D_R(y)$ with $\tau_\alpha(y')=\tau_\alpha(y'')=x'$, and from $\{\Psi_t:t\in\R\}\subset\textup{Aut}(Y,R)$, $\overline{\{\Psi_t(y):t\in\R\}}=\tau_\alpha^{-1}(\tilde x)$, and the distality of $\{\Psi_t:t\in\R\}$ on the fibre $\tau_\alpha^{-1}(x')$, it follows that there exist distinct points $\tilde x',\tilde x''\in\mc D_R(\tilde y)\cap\tau_\alpha^{-1}(x')=\bar{\mc O}_R(\tilde y)\cap\tau_\alpha^{-1}(x')$, giving a contradiction.
Therefore the mapping
\begin{eqnarray*}
\varphi:& & X_\alpha\times\tau_\alpha^{-1}(\tilde x)\longrightarrow Y \\
& & (x,y)\mapsto\tau_\alpha^{-1}(x)\cap\bar{\mc O}_R(y)
\end{eqnarray*}
is well-defined, onto, one-to-one, and by $\bar{\mc O}_R(y)=\mc D_R(y)$ for every $y\in\tau_\alpha^{-1}(\tilde x)$ it is also continuous.
Hence the product $X_\alpha\times\tau_\alpha^{-1}(\tilde x)=X_\alpha\times M$ and the space $Y$ are homeomorphic, and we have the conjugation relation that $\varphi^{-1}\circ R\circ\varphi(x,m)=(T_\alpha x,m)$.
Moreover, from $\{\Psi_t:t\in\R\}\subset\textup{Aut}(Y,R)$ follows for every $t\in\R$ the relation that $\varphi^{-1}\circ\Psi_t\circ\varphi(x,m)=(x,\Phi_t(m))$, in which $\{\Phi_t:t\in\R\}$ is the restriction of $\{\Psi_t:t\in\R\}$ on the compact metric space $M=\tau_\alpha^{-1}(\tilde x)$.
\end{proof}

It should be mentioned here that an ordinal $\xi\leq\eta$ with $(Y,S)=(X_\xi,T_\xi)$ not necessarily exists.
In the next step we want to define a function $f_Y:Y\longrightarrow\R$ so that $f_Y\circ\pi_Y$ is cohomologous to $f$, and thereafter we want to study the dynamical properties of the incremental cocycles $(f-f_\alpha\circ\pi_\alpha) (n,x)$, $(f-f_Y\circ\pi_Y) (n,x)$, and $(f_Y-f_\alpha\circ\tau_\alpha) (n,y)$.

\begin{proposition}\label{prop:res}
\begin{enumerate}
\item For an arbitrary sequence $\{(n_k,x_k)\}_{k\geq 1}\subset\Z\times X$ with
\begin{equation*}
d (x_k, T^{n_k} x_k)\to 0\enspace\textup{and}\enspace (f_\alpha\circ\pi_\alpha)(n_k,x_k)\to 0
\end{equation*}
it holds also true that
\begin{equation*}
(f-f_\alpha\circ\pi_\alpha) (n_k,x_k)\to 0
\end{equation*}
as $k\to\infty$.
This assertion implies that
\begin{equation*}
E(f_\alpha\circ\pi_\alpha,f)\cap(\{0\}\times\R)\subset\{(0,0)\} .
\end{equation*}
Moreover, for a point $(x,t)\in X\times\R$ and a sequence $\{m_k\}_{k\geq 1}\subset\Z$ with $\mathbf T_{f_\alpha\circ\tau_\alpha}(x,t)\to(x',t')$ as $k\to\infty$ the sequence $(f-f_\alpha\circ\pi_\alpha) (m_k,x)$ is convergent to a \emph{finite} limit.
\item There exists a continuous function $f_Y:Y\longrightarrow\R$ so that the function $f_Y\circ\pi_Y:X\longrightarrow\R$ is topologically cohomologous to $f$.
In analogy to (i), we have for an arbitrary sequence $\{(n_k,y_k)\}_{k\geq 1}\subset\Z\times Y$ with
\begin{equation*}
d_Y (y_k, S^{n_k} y_k)\to 0\enspace\textup{and}\enspace (f_\alpha\circ\tau_\alpha)(n_k,y_k)\to 0
\end{equation*}
that it holds also true that
\begin{equation}\label{eq:f_Y}
(f_Y-f_\alpha\circ\tau_\alpha) (n_k,y_k)\to 0
\end{equation}
as $k\to\infty$.
\end{enumerate}
\end{proposition}

We shall prove two technical lemmas first.
The first lemma shows that a ``relative'' non-triviality of a cocycle with respect to another cocycle can be lifted over an extension of the compact metric flow.

\begin{lemma}\label{lem:R2}
Let $(X,T)$ be a minimal compact metric flow, let $(Z,R)=\sigma(X,T)$ be a factor, and let $g=(g_1,g_2):Z\longrightarrow\R^2$ be a continuous function.
Suppose that there exists a sequence $\{(n_k,z_k)\}_{k\geq 1}\subset\Z\times Z$ with $d_Z(z_k,R^{n_k}z_k)\to 0$ so that $g_1(n_k,z_k)\to 0$ and $g_2(n_k,z_k)\nrightarrow 0$ as $k\to\infty$.
Then there exists a sequence $\{(m_k,x_k)\}_{k\geq 1}\subset\Z\times X$ so that $d_X(x_k,T^{m_k}x_k)\to 0$ and $(g\circ\sigma)(m_k,x_k)\to(0,\infty)$.
\end{lemma}

\begin{proof}
Suppose at first that $|g_2(n_k,z_k)|\nrightarrow\infty$ and let $z\in Z$ be a cluster point of the sequence $\{z_k\}_{k\geq 1}$.
Then the statement (ii) in Lemma \ref{lem:at} implies that $E(g)$ has an element of the form $(0,c)$ with $c\in\R\setminus\{0\}$, and as $E(g)$ is a closed subspace we can change the sequence $\{(n_k,z_k)\}_{k\geq 1}\subset\Z\times Z$ so that $d_Z(z_k,R^{n_k}z_k)\to 0$, $g_1(n_k,z_k)\to 0$ and $|g_2(n_k,z_k)|\to\infty$ as $k\to\infty$.
For every cluster point $z$ of $\{z_k\}_{k\geq 1}$ it follows that the point $(z,0,\infty)$ is in the $\mathbf R_g$-prolongation of $(z,0,0)$ in $Z\times(\R_\infty)^2$, and by the statement (ii) in Lemma \ref{lem:at} this holds true for every $z'\in Z$. 
Hence by Lemma \ref{lem:o_p} there exists a point $z'\in Z$ so that $(z',0,\infty)$ is in the orbit closure of $(z',0,0)$, and thus there exists also a sequence of integers $\{i_k\}_{k\geq 1}$ so that $(g_1)(i_k,z')\to 0$ and $(g_2)(i_k,z')\nearrow\infty$ as $k\to\infty$.
Now we choose a point $x'\in\sigma^{-1}(z')$, and by the compactness of $X$ there exists a subsequence $\{j_k\}_{k\geq 1}$ of $\{i_k\}_{k\geq 1}$ so that $d_X(T^{j_{k+1}}x',T^{j_k}x')\to 0$ and $(g_1\circ\sigma, g_2\circ\sigma)(j_{k+1}-j_k,T^{j_k}x')\to(0,\infty)$.
The sequence $\{(m_k,x_k)\}_{k\geq 1}\subset\Z\times X$ with the required properties is given by $\{(m_k,x_k)=(j_{k+1}-j_k,T^{j_k}x')\}_{k\geq 1}$.
\end{proof}

The second lemma shows that an action by a group of automorphisms extending $(X_\alpha,T_\alpha)$ and a cocycle which is ``relatively'' trivial with respect to $f_\alpha$ give rise to a cocycle of the joint action of $\Z$ and the group of automorphisms.

\begin{lemma}\label{lem:tilde_coc}
Let $(Z,R)$ be a distal minimal compact metric flow which extends $(X_\alpha,T_\alpha)=\sigma_\alpha(Z,R)$, and let $G\subset\textup{Aut}(Z,R)$ be a Hausdorff topological group preserving the fibres of $\sigma_\alpha$.
Suppose that there exists a topological group homomorphism $\varphi:G\longrightarrow\R$ so that for every $h\in G$ and every $z\in Z$ it holds true that
\begin{equation*}
(h(z),\varphi(h))\in\mc D_{R,f_\alpha\circ\sigma_\alpha}(z,0) .
\end{equation*}
Furthermore, suppose that $g:Z\longrightarrow\R$ is a continuous function so that for every sequence $\{(n_k,z_k)\}_{k\geq 1}\subset\Z\times Z$ with $d_Z(z_k,R^{n_k}z_k)\to 0$ and $(f_\alpha\circ\sigma_\alpha)(n_k,z_k)\to 0$ as $k\to\infty$, it holds also true that
\begin{equation}\label{eq:g}
g(n_k,z_k)\to 0.
\end{equation}
Then there exists a topological cocycle $\tilde g((n,h),z)$ of the $\Z\times G$-action $\{h\circ R^n:(n,h)\in\Z\times G\}$ so that
\begin{equation*}
g(n,z)=\tilde g((n,\mathbf 1_G),z)
\end{equation*}
holds true for every $z\in Z$ and every integer $n$.
\end{lemma}

\begin{proof}
We let $\mathbf U:Z\times(\R_\infty)^2\longrightarrow Z\times(\R_\infty)^2$ be the skew product transformation defined by $(Z,R)$ and the function $(f_\alpha\circ\sigma_\alpha,g)$, and we fix a point $\tilde z\in Z$ so that the $\bar{\mc O}_\mathbf U (\tilde z,0,0)$ and $\mc D_\mathbf U (\tilde z,0,0)$ coincide in $Z\times(\R_\infty)^2$ (cf. Lemma \ref{lem:o_p}).
Then we select a sequence $\{n_k^h\}_{k\geq 1}$ for every $h\in G$ so that $(\mathbf R_{f_\alpha\circ\sigma_\alpha})^{n_k^h}(\tilde z,0)\to(h(\tilde z),\varphi(h))$ as $k\to\infty$.
As $h$ is an automorphism of $(Z,R)$ and the function $f_\alpha\circ\sigma_\alpha$ is invariant under $h$, we can conclude for every fixed integer $m$ (cf. the proof of Proposition \ref{prop:flow}) that $(\mathbf R_{f_\alpha\circ\sigma_\alpha})^{n_k^h}(R^m \tilde z,0)\to(h(R^m \tilde z),\varphi(h))$ as $k\to\infty$.
By equation (\ref{eq:g}) the sequence $\{g(n+n_k^h,R^m\tilde z)\}_{k\geq 1}$ is convergent for every integer $n$, and hence we can put
\begin{equation}\label{eq:def_g}
\tilde g((n,h),R^m\tilde z)=\lim_{k\to\infty} g(n+n_k^h,R^m\tilde z) .
\end{equation}
Moreover, it follows from equation (\ref{eq:g}) that the definition of $\tilde g((n,h),R^m\tilde z)$ is independent of the choice of the sequence $\{n_k^h\}_{k\geq 1}$.
We claim that this mapping extends from the $R$-orbit of $\tilde z$ to a continuous mapping $\tilde g:\Z\times G\times Z\longrightarrow\R$.
If not, then there exist a point $(n,h,z)\in\Z\times G \times Z$ and sequences $\{m_k^{(i)}\}_{k\geq 1}$ and $\{n_k^{(i)}\}_{k\geq 1}$ with $i\in\{1,2\}$ so that $R^{m_k^{(i)}}\tilde z\to z$, $R^{m_k^{(i)}+n_k^{(i)}}\tilde z\to h(z)$, and
\begin{equation*}
(f_\alpha\circ\sigma_\alpha)(n+n_k^{(i)},R^{m_k^{(i)}}\tilde z)\to\varphi(h)+(f_\alpha\circ\sigma_\alpha)(n,z) ,
\end{equation*}
as $k\to\infty$, while the limit points
\begin{equation*}
\tilde g_i=\lim_{k\to\infty} g(n+n_k^{(i)},R^{m_k^{(i)}}\tilde z)\in\R_\infty
\end{equation*}
are either distinct for $i\in\{1,2\}$ or both of them are equal to $\infty$.
For $i\in\{1,2\}$ the point $(h(R^n z),\varphi(h)+(f_\alpha\circ\sigma_\alpha)(n,z),\tilde g_i)$ is an element of $\mc D_\mathbf U (z,0,0)$, and then it follows for every integer $m$ that the point
\begin{eqnarray*}
(h(R^{m+n}z),\varphi(h)+(f_\alpha\circ\sigma_\alpha)(n,z)+(f_\alpha\circ\sigma_\alpha)(m,h(R^n z))-(f_\alpha\circ\sigma_\alpha)(m,z),\ \\
g(m,h(R^n z))+\tilde g_i-g(m,z))=\\
=(h(R^{m+n} z),\varphi(h)+(f_\alpha\circ\sigma_\alpha)(n,R^m z),g(m,h(R^n z))+\tilde g_i-g(m,z))
\end{eqnarray*}
is an element of $\mc D_\mathbf U (R^m z,0,0)$.
Hence by the density of the $R$-orbit of $z$ and $h\in\textup{Aut}(Z,R)$ there are either distinct elements $a_1,a_2\in\R_\infty$ with
\begin{equation*}
(h(R^n \tilde z),\varphi(h)+(f_\alpha\circ\sigma_\alpha)(n,\tilde z),a_i)\in\mc D_\mathbf U (\tilde z,0,0)
\end{equation*}
or it holds that $(h(R^n \tilde z),\varphi(h)+(f_\alpha\circ\sigma_\alpha)(n,\tilde z),\infty)\in\mc D_\mathbf U (\tilde z,0,0)$.
In either case occurs a contradiction to equality (\ref{eq:g}) by the coincidence of the sets $\bar{\mc O}_\mathbf U (\tilde z,0,0)$ and $\mc D_\mathbf U (\tilde z,0,0)$ in $Z\times(\R_\infty)^2$.

Now we turn to the cocycle identity.
Let $(j_i,h_i)\in\Z\times G$ with $i\in\{1,2\}$ be two arbitrary elements, and choose sequences $\{n_k^{h_i}\}_{k\geq 1}\subset\Z$ as above.
We fix an integer $m$ and conclude from equality (\ref{eq:def_g}) and $h_2\in\textup{Aut}(Z,R)$ that
\begin{eqnarray*}
\tilde g((j_1,h_1),h_2(R^{j_2+m}\tilde z))+\tilde g((j_2,h_2),R^m\tilde z)=\hspace{4.5cm}\\
=\lim_{l\to\infty}\lim_{k\to\infty}g(j_1+n_l^{h_1},R^{j_2+n_k^{h_2}+m}\tilde z)+\lim_{k\to\infty}g(j_2+n_k^{h_2},R^{m}\tilde z)=\\
=\lim_{l\to\infty}\lim_{k\to\infty}g(j_1+n_l^{h_1}+j_2+n_k^{h_2},R^{m}\tilde z) .
\end{eqnarray*}
A diagonalisation of sequences and the invariance $f_\alpha\circ\sigma_\alpha=f_\alpha\circ\sigma_\alpha\circ h_2$ imply that
\begin{equation*}
(\mathbf R_{f_\alpha\circ\sigma_\alpha})^{n_l^{h_1}+n_{k_l}^{h_2}}(\tilde z,0)\to((h_1( h_2(\tilde z)),\varphi(h_1)+\varphi(h_2))=((h_1 h_2)(\tilde z),\varphi(h_1 h_2))
\end{equation*}
and
\begin{eqnarray*}
\lim_{l\to\infty}\lim_{k\to\infty}g(j_1+n_l^{h_1}+j_2+n_k^{h_2},R^{m}\tilde z)=\hspace{3cm}\\
\lim_{l\to\infty}g(j_1+n_l^{h_1}+j_2+n_{k_l}^{h_2},R^{m}\tilde z)=\tilde g((j_1+j_2,h_1+h_2),R^m\tilde z) .
\end{eqnarray*}
From $\bar{\mc O}_R(\tilde z)=Z$ and the continuity of $\tilde g$ on $\Z\times G\times Z$ follows the cocycle equality $\tilde g((j_1,h_1),h_2(R^{j_2} z))+\tilde g((j_2,h_2),z)=g((j_1+j_2,h_1+h_2),z)$ for all $z\in Z$.
\end{proof}

\begin{proof}[Proof of Proposition \ref{prop:res}]
(i)
We let $\{(n_k,x_k)\}_{k\geq 1}\subset\Z\times X$ be a fixed sequence with $d (x_k, T^{n_k} x_k)\to 0$ and $(f_\alpha\circ\pi_\alpha)(n_k,x_k)\to 0$ as $k\to\infty$, and we want to verify by transfinite induction that $f(n_k,x_k)\to 0$.

Suppose that $\xi$ is an ordinal with $\alpha\leq\xi<\eta$ and put $\gamma=\xi+1$, and suppose that $(f_\xi\circ\pi_\xi)(n_k,x_k)\to 0$ and $(f_\gamma\circ\pi_\gamma)(n_k,x_k)\nrightarrow 0$ as $k\to\infty$.
Let $(\tilde X,\tilde T)$ and $K\subset\textup{Aut}(\tilde X,\tilde T)$ define a compact metric group extension of $(X_\xi, T_\xi)$ so that $(X_\gamma,T_\gamma)$ is the $H$-orbit space in $(\tilde X,\tilde T)$ of a compact subgroup $H\subset K$, and denote by $\sigma$ the factor map from $(\tilde X,\tilde T)$ to $(X_\gamma,T_\gamma)$.
Now consider the $\R^2$-valued cocycle of $(X_\gamma,T_\gamma)$ defined by the function $g=(f_\xi\circ\pi_\xi^\gamma, f_\gamma-f_\xi\circ\pi_\xi^\gamma)$.
By the Lemma \ref{lem:R2} there exist a sequence $(m_k,\tilde x_k)\in\Z\times\tilde X$ with $\tilde d(\tilde x_k,\tilde T^{m_k}\tilde x_k)\to 0$ and
\begin{equation*}
(f_\xi\circ\pi_\xi^\gamma\circ\sigma, f_\gamma\circ\sigma-f_\xi\circ\pi_\xi^\gamma\circ\sigma)(m_k,\tilde x_k)\to(0,\infty) .
\end{equation*}
As the skew product extension $\mathbf T_{\gamma,f_\gamma}$ is not topologically transitive due to the maximality of the ordinal $\alpha$, also $\mathbf{\tilde T}_{f_\gamma\circ\sigma}$ is not topologically transitive.
By Lemma \ref{lem:at} there exist a compact neighbourhood $L\subset\R$ of $0$ and an $\varepsilon>0$ so that $\tilde d_\gamma(\tilde x, \tilde T ^n \tilde x)<\varepsilon$ for some $\tilde x\in\tilde X$ and $n\in\Z$ implies that $(f_\gamma\circ\sigma)(n,\tilde x)\notin 2L\setminus L^0$, and by Lemma \ref{lem:iso} there exists a $\delta>0$ so that $\tilde d_\gamma(\tilde x,\tilde T ^n \tilde x)<\delta$ for $\tilde x\in\tilde X$ and $n\in\Z$ implies that $\tilde d_\gamma(g(\tilde x),\tilde T ^ng(\tilde x))<\varepsilon$ for every $g\in K$.
As the cocycle $(f_\gamma-f_\xi\circ\pi_\xi^\gamma)(m_k,x_\gamma)$ has a connected range and a zero on the fibre $(\pi_\xi^\gamma)^{-1}(\pi_\xi^\gamma\circ\sigma(\tilde x_k))$, a contradiction arises for all large enough $k\geq 1$ with the properties that $\tilde d(\tilde x_k,\tilde T^{m_k}\tilde x_k)<\delta$, $(f_\xi\circ\pi_\xi^\gamma\circ\sigma)(m_k,\tilde x_k)\in L$, and $(f_\gamma\circ\sigma-f_\xi\circ\pi_\xi^\gamma\circ\sigma)(m_k,\tilde x_k)\notin 3 L$.
Therefore we can conclude that $(f_\gamma\circ\pi_\gamma)(n_k,x_k)\rightarrow 0$ as $k\to\infty$.

Now suppose that $\gamma$ is a limit ordinal and that $(f_\xi\circ\pi_\xi)(n_k,x_k)\to 0$ holds true for all ordinals $\xi$ with $\alpha\leq\xi<\gamma$.
From the maximality of the ordinal $\alpha$ it follows again that $\mathbf T_{\gamma,f_\gamma}$ is not topologically transitive, and thus there exist (cf. Lemma \ref{lem:at}) a compact neighbourhood $L\subset\R$ of zero and an $\varepsilon>0$ so that $d_\gamma(x, T_\gamma ^n x)<\varepsilon$ for $x\in X_\gamma$ and $n\in\Z$ implies that $f_\gamma(n,x)\notin 2L\setminus L^0$.
Given $\varepsilon$ we can find an ordinal $\alpha<\zeta<\gamma$ so that $d_\gamma (x,z)<\varepsilon/3$ holds true for all $x,z\in X_\gamma$ with $\pi_\zeta^\gamma (x)=\pi_\zeta^\gamma (z)$, and for all integers $k$ with $d_\gamma(\pi_\gamma(x_k),\pi_\gamma(T^{n_k} x_k))<\varepsilon/3$ this implies that $d_\gamma(z_k, T_\gamma^{n_k} z_k)<\varepsilon$ for all $z_k\in X_\gamma$ with $\pi_\zeta^\gamma (z_k)=\pi_\zeta(x_k)$.
The cocycle $(f_\gamma-f_\zeta\circ\pi_\zeta^\gamma)(n,x)$ has a zero $z_k$ on every fibre $(\pi_\zeta^\gamma)^{-1}(\pi_\zeta(x_k))$, and it follows for every $k\geq 1$ with $(f_\zeta\circ\pi_\zeta)(n_k,x_k)\in L$ from the inclusion $f_\gamma(n_k,z_k)\notin 2L\setminus L^0$ that $f_\gamma(n_k,z_k)\in L$.
We can conclude from the connectedness of the fibre that $(f_\gamma\circ\pi_\gamma)(n_k,x_k)\in L$, and as the neighbourhood $L$ can be selected arbitrarily small, it follows that $(f_\gamma\circ\pi_\gamma)(n_k,x_k)\to 0$ as $k\to\infty$.

The transfinite induction process gives $f(n_k,x_k)\to 0$ as $k\to\infty$, and together with $(f_\alpha\circ\pi_\alpha)(n_k,x_k)\to 0$ also $(f-f_\alpha\circ\pi_\alpha)(n_k,x_k)\to 0$ as $k\to\infty$.

(ii)
We let $(Y_c,S_c)=\pi_c(X,T)$ be the flow defined by the connected components of the fibres of $\pi_Y$ (cf. \cite{MMWu}, Definition 2.3), and we let $\rho$ be the factor map from $(Y_c,S_c)$ onto $(Y,S)=\rho(Y_c,S_c)$.
There exists a family of relatively invariant probability measures $\{\mu_{c,y}:y\in Y_c\}$ with $\mu_{c,y}(\pi_c^{-1}(y))=1$ for every $y\in Y_c$ and the properties stated in Proposition \ref{prop:meas}, and we can define a continuous function $f_c:Y_c\longrightarrow\R$ by
\begin{equation*}
f_c(y)=\mu_{c,y}(f)
\end{equation*}
so that for every $y\in Y_c$ every integer $n$ it holds that $f_c(n,y)=\mu_{c,y}(f(n,\cdot))$.
We select a point $\tilde x\in X$ with $\mc D_{T,f_\alpha\circ\pi_\alpha}(T^n \tilde x,0)=\bar{\mc O}_{T,f_\alpha\circ\pi_\alpha}(T^n \tilde x,0)$ for every integer $n$, and we want to verify that the cocycle $(f-f_c\circ\pi_c)(n,\tilde x)$ is uniformly bounded for all integers $n$ and thus is a topological coboundary (cf. \cite{LM}, Lemma 3.1).
The construction of the factor $(Y,S)$ (cf. equality (\ref{eq:G})) and $\mc D_{T,f_\alpha\circ\pi_\alpha}(T^n \tilde x,0)=\bar{\mc O}_{T,f_\alpha\circ\pi_\alpha}(T^n \tilde x,0)$ imply for every integer $n$ that
\begin{equation*}
\bar{\mc O}_{T,f_\alpha\circ\pi_\alpha}(\tilde x,0)\cap(\pi_c^{-1}(\pi_c(T^n\tilde x))\times\R)=\pi_c^{-1}(\pi_c(T^n\tilde x))\times\{(f_\alpha\circ\pi_\alpha)(n,\tilde x)\} .
\end{equation*}
We let $\mathbf U:X\times\R^2\longrightarrow X\times\R^2$ be the skew product transformation defined by $(X,T)$ and the function $(f_\alpha\circ\pi_\alpha,f)$, and we observe that
\begin{eqnarray*}
\bar{\mc O}_\mathbf U (\tilde x,0,0)\cap(\pi_c^{-1}(\pi_c(T^n\tilde x))\times\{(f_\alpha\circ\pi_\alpha)(n,\tilde x)\}\times\R)=\\
\{(x,(f_\alpha\circ\pi_\alpha)(n,\tilde x),\phi_n(x)):x\in\pi_c^{-1}(\pi_c(T^n\tilde x))\}
\end{eqnarray*}
holds for every integer $n$, in which $\phi_n:\pi_c^{-1}(\pi_c(T^n\tilde x))\longrightarrow\R$ is a continuous function.
Indeed, from $T^{n+n_k}\tilde x\to x\in\pi_c^{-1}(\pi_c(T^n\tilde x))$ and $(f_\alpha\circ\pi_\alpha)(n_k,T^n \tilde x)\to 0$ it follows by assertion (i) that $f(n_k,T^n \tilde x)$ is convergent and that this limit is unique, and the corresponding set in the orbit closure is the closed graph of $\phi_n$.
Furthermore, assertion (i) implies that for every $\varepsilon>0$ there exists a $\delta>0$ so that $x,x'\in\pi_c^{-1}(\pi_c(T^n\tilde x))$ and $d(x,x')<\delta$ are sufficient for $|\phi_n(x)-\phi_n(x')|<\varepsilon$, uniformly in $n$.
Hence by the connectedness of the fibres of $\pi_c$ there exists a constant $D>0$ with $|\phi_n(x)-\phi_n(x')|<D$ for all integers $n$ and all $x,x'\in\pi_c^{-1}(\pi_c(T^n\tilde x))$, and as $f_c(n,\pi_c(\tilde x))$ is defined by the $\mu_{c,\pi_c(\tilde x)}$-integral of $f(n,x)$, it follows that $|(f-f_c\circ\pi_c)(n,\tilde x)|< 2D$ for all integers $n$.
Moreover, we can conclude that assertion (ii) is valid with respect to $f_c$, $(Y_c,S_c)$, $\pi_c$, and $\tau_\alpha\circ\rho$.
Indeed, if there exists a sequence $\{(n_k,y_k)\}_{k\geq 1}\subset\Z\times Y_c$ so that $d_c(y_k, S_c^{n_k} y_k)\to 0$, $(f_\alpha\circ\tau_\alpha\circ\rho)(n_k,y_k)\to 0$, and $f_c(n_k,y_k)\nrightarrow 0$ as $k\to\infty$, then Lemma \ref{lem:R2} and the boundedness of the transfer function between $f$ and $f_c\circ\pi_c$ give a contradiction to assertion (i).

Now we consider the extension from $(Y,S)$ to $(Y_c,S_c)$.
By Theorem 3.7 in \cite{MMWu} this is an isometric extension, and by Fact \ref{fact:iso} there exists a compact metric group extension $(\tilde Y,\tilde S)$ of $(Y,S)$ by $K\subset\textup{Aut}(\tilde Y,\tilde S)$ so that $(Y_c,S_c)=\sigma(\tilde Y,\tilde S)$ is the factor defined by the orbit space of a compact subgroup $H\subset K$.
Moreover, we put $\tilde\sigma=\rho\circ\sigma$ so that $(Y,S)=\tilde\sigma(\tilde Y,\tilde S)$.
By the properties of $f_c$, for every sequence $\{(n_k,\tilde y_k)\}_{k\geq 1}\subset\Z\times\tilde Y$ with $d_{\tilde Y} (\tilde y_k,\tilde S^{n_k}\tilde y_k)\to 0$ and $(f_\alpha\circ\tau_\alpha\circ\tilde\sigma)(n_k,\tilde y_k)\to 0$ also $(f_c\circ\sigma)(n_k,\tilde y_k)\to 0$.
Thus we can apply Lemma \ref{lem:tilde_coc} for the joint action of $\tilde S$ and $K$ and the homomorphism $\varphi\equiv 0$, and we obtain a real valued cocycle $\tilde f_c((n,h),\tilde y)$ with $(n,h)\in\Z\times K$, $\tilde y\in\tilde Y$, so that for every integer $n$ and every $\tilde y\in\tilde Y$ it holds true that $\tilde f_c((n,\mathbf 1_K),\tilde y)=(f_c\circ\sigma)(n,\tilde y)$.
We let the continuous function $f_Y:Y\longrightarrow\R$ be defined by the integral of $f_c\circ\sigma$ over the $K$-orbits in $(\tilde Y,\tilde S)$ with respect to the Haar measure on $K$, and we observe that the integral of $(f_c\circ\sigma)(n,\cdot)$ over the $K$-orbit of $\tilde y$ is equal to $f_Y(n,\tilde\sigma(\tilde y))$ for every $\tilde y\in\tilde Y$ and every integer $n$.
From the cocycle identity for the action of $\Z\times K$ and the uniform boundedness of $\tilde f_c((0,h),\tilde y)$ for all $(h,\tilde y)\in K\times\tilde Y$ we can conclude that $(f_c-f_Y\circ\rho)\circ\sigma:\tilde Y\longrightarrow\R$ is a topological coboundary of $\tilde S$.
Therefore also the cocycle $(f_c-f_Y\circ\rho)(n,y)$ is uniformly bounded for all integers $n$ and all $y\in Y_c$, whence it is as well a topological coboundary.
The convergence (\ref{eq:f_Y}) follows now by Lemma \ref{lem:R2} and the boundedness of the transfer function between $f_c$ and $f_Y\circ\rho$.
\end{proof}

\begin{proposition}\label{prop:inc}
The cocycle $(f_Y-f_\alpha\circ\tau_\alpha)(n,y)$ can be extended  to a topological cocycle $\tilde f((n,t),y)$ of the $\Z\times\R$-flow $\{\Psi_t\circ S^n:(n,t)\in\Z\times\R\}$ in the sense that
\begin{equation*}
(f_Y-f_\alpha\circ\tau_\alpha)(n,y)=\tilde f((n,0),y)
\end{equation*}
for every $y\in Y$ and every integer $n$.
Moreover, there exists a continuous function $b:Y\longrightarrow\R$ so that for every $y\in Y$ and every integer $n$ it holds true that
\begin{equation}\label{eq:tf}
\tilde f((n,-(f_\alpha\circ\tau_\alpha)(n,y)),y)=b(\Psi_{-(f_\alpha\circ\tau_\alpha)(n,y)} (S^n y))-b(y) =b(R^n y)-b(y) ,
\end{equation}
and therefore the continuous function
\begin{equation*}
y\mapsto\tilde f((1,-(f_\alpha\circ\tau_\alpha)(y)),y)
\end{equation*}
is a topological coboundary with transfer function $b:Y\longrightarrow\R$ over the distal homeomorphism $R:Y\longrightarrow Y$ with $Ry=\Psi_{-(f_\alpha\circ\tau_\alpha)(y)}(Sy)$.
\end{proposition}

\begin{proof}
The assertion (ii) of Proposition \ref{prop:res} shows that the function $g=f_Y-f_\alpha\circ\tau_\alpha$, the group $G=\{\Psi_t:t\in\R\}\subset\textup{Aut}(Y,S)$, and the group homomorphism $\varphi=\textup{id}_\R$ fulfil the requirements of Lemma \ref{lem:tilde_coc}.
We obtain a cocycle $\tilde f((n,t),y)$ extending $(f_Y-f_\alpha\circ\tau_\alpha)(n,y)$, and it remains to construct a continuous function $b:Y\longrightarrow\R$ so that equality (\ref{eq:tf}) holds true for every point $y\in Y$ and every integer $n$.

Let $\mathbf U:Y\times(\R_\infty)^2\longrightarrow Y\times(\R_\infty)^2$ be the skew product transformation defined by $(Y,S)$ and the function $(f_\alpha\circ\tau_\alpha,f_Y-f_\alpha\circ\tau_\alpha)$, and choose a point $\tilde x\in X_\alpha$ so that there exists a point $\tilde y\in\tau_\alpha^{-1}(\tilde x)$ with $\bar{\mc O}_\mathbf U(\tilde y,0,0)=\mc D_\mathbf U(\tilde y,0,0)$ as well as $\bar{\mc O}_{T_\alpha,f_\alpha}(T_\alpha^n \tilde x,0)=X_\alpha\times\R$ holds true for every integer $n$.
We want to verify at first that for every $y\in Y$ with $\tau_\alpha(y)=\tilde x$ and every $y'\in Y$ the set
\begin{equation*}
C_{(y,y')}=\mc D_\mathbf U(y,0,0)\cap(\{y'\}\times\{0\}\times\R_\infty)
\end{equation*}
has at most one element.
Suppose that $(y',0,s_i)\in C_{(y,y')}$ with distinct $s_i\in\R_\infty$ for $i\in\{1,2\}$, then $\Psi_t\in\textup{Aut}(Y,S)$, the properties of $\tilde x$, and equality (\ref{eq:o_flow}) imply for every $t\in\R$ that
\begin{equation*}
(\Psi_t(y'),0,s_i+\tilde f((0,t),y')+\tilde f((0,-t),y))\in C_{(\Psi_t(y),\Psi_t(y'))} .
\end{equation*}
By the minimality of $\{\Psi_t:t\in\R\}$ on the fibre $\tau_\alpha^{-1}(\tilde x)$, there exists a point $\tilde y'\in Y$ so that the set $C_{(\tilde y,\tilde y')}$ contains either two distinct points or the point $(\tilde y',0,\infty)$.
In either case occurs a contradiction to Proposition \ref{prop:res} (ii), due to the identity $\bar{\mc O}_\mathbf U(\tilde y,0,0)=\mc D_\mathbf U(\tilde y,0,0)$.

For every $y\in Y$ with $\tau_\alpha(y)=\tilde x$ and every $y'\in Y$ there exists exactly one point in the set $\bar{\mc O}_{S,f_\alpha\circ\tau_\alpha}(y,0)\cap(\tau_\alpha^{-1}(\tau_\alpha(y'))\times\{0\})$, and by the distality of the extension from $(X_\alpha,T_\alpha)$ to $(Y,S)$ it follows for distinct $y_1, y_2\in\tau_\alpha^{-1}(\tilde x)$ that these points in $(\tau_\alpha^{-1}(\tau_\alpha(y'))\times\{0\})$ are distinct.
Therefore we can define a function $b:Y\longrightarrow\R$ with $b(y)=0$ for every $y\in\tau_\alpha^{-1}(\tilde x)$ by
\begin{equation*}
b(y')=\{t\in\R:(y',0,t)\in\cup_{y\in\tau_\alpha^{-1}(\tilde x)}\bar{\mc O}_\mathbf U (y,0,0)\cap(Y\times\{0\}\times\R)\} .
\end{equation*}
Indeed, for every $y\in\tau_\alpha^{-1}(\tilde x)$ the projection of the set $\bar{\mc O}_\mathbf U (y,0,0)\cap(Y\times\{0\}\times\R)$ on the first coordinate is exactly the orbit closure $\bar{\mc O}_R(y)$, and by the distality of $R$ the set of orbit closures forms a partition of $Y$.
The continuity of the function follows because $C_{(y,y')}$, which is defined with $\mc D_\mathbf U(y,0,0)$ instead of $\bar{\mc O}_\mathbf U (y,0,0)$, has at most one element for every pair $(y,y')\in Y^2$.

The equality (\ref{eq:tf}) follows now easily.
The function on $\Z\times Y$ defined by
\begin{equation*}
(n,y)\mapsto\tilde f((n,-(f_\alpha\circ\tau_\alpha)(n,y)),y)=\sum_{k=0}^{n-1}\tilde f((1,-f_\alpha\circ\tau_\alpha(R^k y)),R^k y)
\end{equation*}
is the cocycle given by the $\Z$-action $(n,y)\mapsto R^n y$ and the continuous function $y\mapsto\tilde f((1,-f_\alpha\circ\tau_\alpha(y)),y)$.
By the construction of the function $b:Y\longrightarrow\R$ this cocycle is a topological coboundary with $b$ as its transfer function.
\end{proof}

With these prerequisites we can conclude the proof of the Main Theorem:

\begin{proof}[Proof of the Main Theorem]
We let all the elements of the statement and the flow $\{\Psi_t:t\in\R\}\subset\textup{Aut}(Y,R)\cap\textup{Aut}(Y,S)$ be defined according to Propositions \ref{prop:max}, \ref{prop:flow}, \ref{prop:res}, and \ref{prop:inc}.
We define the cocycle $g(t,m)$ of the flow $\{\Phi_t:t\in\R\}$ for all $m\in M$ and $t\in\R$ by
\begin{equation*}
g(t,m)=\tilde f((0,t),(\tilde x,m))
\end{equation*}
with $\tilde x\in X_\alpha$ chosen so that $b(y)=0$ for all $y\in\tau_\alpha^{-1}(\tilde x)$ (cf. the proof of Proposition \ref{prop:inc}).
It follows then for arbitrary $y=(x,m)\in Y$ and $t\in\R$ that
\begin{equation*}
\tilde f((0,t),(x,m))= g(t,m)+b(x,\Phi_t(m))-b(x,m) ,
\end{equation*}
because the function $y\mapsto\tilde f((1,-(f_\alpha\circ\tau_\alpha)(y)),y)$ is a topological coboundary with transfer function $b:Y\longrightarrow\R$ over the homeomorphism $R$ with $\{\Psi_t:t\in\R\}\subset\textup{Aut}(Y,R)$.
The cocycle identity for $\tilde f((n,t),y)$ and the equality (\ref{eq:tf}) imply for every $(x,m)\in Y$ that
\begin{eqnarray}\label{eq:f_Yn}
f_Y(x,m)-f_\alpha(x)=\tilde f((0,f_\alpha(x)),R(x,m))+b(R(x,m))-b(x,m)=\nonumber \\
g(f_\alpha(x),m)+b(T_\alpha x,\Phi_{f_\alpha(x)}(m))-b(x,m)=g(f_\alpha(x),m)+(b\circ S-b)(x,m).
\end{eqnarray}
The assertion of the theorem follows now by replacing the function $f_Y$ by the cohomologous function $f_Y-(b\circ S-b)$, and the inclusion $\pi_Y^{-1}(x)\times\{0\}\subset\bar{\mc O}_{T, f_Y\circ\pi_Y}(x,0)$ for all $x$ in a residual subset of $X$ follows now from inclusion (\ref{eq:o_X}).
\end{proof}

\begin{proof}[Proof of the Corollary to the Main Theorem]
We let $F:Y\times\R\longrightarrow Y\times\R$ be the map defined in the Corollary, and we conclude for every $(x,m,t)\in Y\times\R$ that
\begin{eqnarray*}
F\circ\mathbf R_{f_\alpha\circ\tau_\alpha}(x,m,t)=F(T_\alpha x,m,t+f_\alpha(x))=\hspace{4.8cm}\\
=(T_\alpha x,\Phi_{t+f_\alpha(x)}(m),t+f_\alpha(x)+g(f_\alpha(x)+t,m))=\\
=(T_\alpha x,\Phi_{t+f_\alpha(x)}(m),t+f_Y(x,\Phi_t(m))+g(t,m))=\mathbf S_{f_Y}\circ F(x,m,t) .
\end{eqnarray*}
If the compact metric flow $(X,T)$ is uniquely ergodic, then its factors $(Y,S)$ and $(X_\alpha,T_\alpha)$ are also uniquely ergodic.
This however forces also the  unique ergodicity of the $\R$-flow $(M,\{\Phi_t:t\in\R\})$, because every $\{\Phi_t:t\in\R\}$-invariant probability measure on $M$ gives rise to an $S$-invariant measure on $Y$ defined as product measure on $X_\alpha\times M$.
This unique $\{\Phi_t:t\in\R\}$-invariant probability measure on $M$ coincides therefore with the relatively invariant measure on the fibres of $\tau_\alpha$.
The function $(f_Y-f_\alpha\circ\tau_\alpha)(n,\cdot)$ has zero integral on every $\tau_\alpha$-fibre for every integer $n$, and we can conclude from equation (\ref{eq:f_Yn}) that $g(t,\cdot)/t\to 0$ uniformly as $|t|\to\infty$.
For an arbitrary point $(x,m,t)\in Y\times\R$ and for every $s\in\R$ giving a solution to the equation
\begin{eqnarray*}
s+g(s,\Phi_{-s}(m))=s-g(-s,m)=t
\end{eqnarray*}
it holds true that $F(x,\Phi_{-s}(m),s)=(x,m,t)$.
Such a solution always exists, because by the uniform convergence $g(t,\cdot)/t\to 0$ the mapping $s\mapsto s+b(x,m)-g(-s,m)$ is onto $\R$ for every choice of $(x,m)\in Y$.
Moreover, we can conclude that every solution $s$ fulfils that $|s-t|<L$ with a uniform constant $L$ for all $(x,m,t)\in Y\times\R$.
From the compactness of $Y$ it follows now easily that the mapping $F$ is closed.
\end{proof}

\textbf{Acknowledgement}: The author would like to thank Professor Jon Aaronson and Professor Eli Glasner for useful discussions and encouragement while this work was done.

\end{document}